\documentclass[12pt]{article}%
\usepackage{amssymb}
\usepackage{amsfonts}
\usepackage{amsmath}
\usepackage{color}
\usepackage{graphicx}%
\providecommand{\U}[1]{\protect\rule{.1in}{.1in}}
\newtheorem{theorem}{Theorem}

\newtheorem{corollary}[theorem]{Corollary}

\newtheorem{definition}[theorem]{Definition}

\newtheorem{lemma}[theorem]{Lemma}

\newtheorem{proposition}[theorem]{Proposition}
\newtheorem{remark}[theorem]{Remark}

\newenvironment{proof}[1][Proof]{\noindent\textbf{#1.} }{\ \rule{0.5em}{0.5em}}
\numberwithin{equation}{section}
\DeclareMathOperator{\argmin}{argmin}

\begin{document}

\title{A variational approach to symmetry, monotonicity, and comparison for
doubly-nonlinear equations\thanks{\textbf{Acknowledgment.}\quad\textrm{This
work has been supported by the Austrian Science Fund (FWF) project P27052-N25.
The author would like to acknowledge the kind hospitality of the Erwin
Schr\"{o}dinger International Institute for Mathematics and Physics, where
part of this research was developed under the frame of the Thematic Program
\textit{Nonlinear Flows}.}}}
\author{{Stefano Melchionna\thanks{ Faculty of Mathematics, University of Vienna,
Oskar-Morgenstern-Platz 1, 1090 Wien, Austria. $\qquad$ E-mail:
\texttt{stefano.melchionna@univie.ac.at}} }}
\maketitle

\begin{abstract}
We advance a variational method to prove qualitative properties such as
symmetries, monotonicity, upper and lower bounds, sign properties, and
comparison principles for a large class of doubly-nonlinear evolutionary
problems including gradient flows, some nonlocal problems, and systems of
nonlinear parabolic equations.

Our method is based on the so-called Weighted-Energy-Dissipation (WED)
variational approach. This consists in defining a global parameter-dependent
functional over entire trajectories and proving that its minimizers converge
to solutions to the target problem as the parameter goes to zero. Qualitative
properties and comparison principles can be easily proved for minimizers of
the WED functional and, by passing to the limit, for the limiting problem.

Several applications of the abstract results to systems of nonlinear PDEs and
to fractional/nonlocal problems are presented. Eventually, we present some
extensions of this approach in order to deal with rate-independent systems and
hyperbolic problems.

\end{abstract}


\vspace{2mm}

\noindent\textbf{AMS (MOS) subject clas\-si\-fi\-ca\-tion:} 35B06, 35B51, 49J27.

\newpage

\section{Introduction}

In this paper we illustrate a general procedure to prove qualitative
properties and comparison principles for the abstract doubly-nonlinear system
given by%
\begin{align}
\mathrm{d}_{V}\psi(u^{\prime})+\eta^{1}-\eta^{2}-f(u) &  =0\text{ a.e. in
}(0,T)\text{,}\label{nonconvex target equation}\\
\eta^{1} &  \in\partial\varphi^{1}(u)\text{, }\eta^{2}\in\partial\varphi
^{2}(u)\text{,}\label{non conv targ 2}\\
u(0) &  =u_{0}\text{.}\label{ic}%
\end{align}
Here $u^{\prime}$ denotes the time derivative of the unknown trajectory
$t\in(0,T)\longmapsto u(t)\in V$, $\varphi^{1},\varphi^{2},\psi$ are proper,
lower semicontinuous, and convex functionals on a Banach space $V$,
$\partial\varphi^{1}$ and $\partial\varphi^{2}$ denote the subdifferentials of
$\varphi^{1}$ and $\varphi^{2}$ respectively, $\mathrm{d}_{V}\psi$ is the
Fr\'{e}chet differential of $\psi$, and $f:V\rightarrow V^{\ast}$ is a
continuous map. Note that we do not assume a differential structure on $f$,
thus $f$ is \textit{nonpotential}.

The abstract system (\ref{nonconvex target equation})-(\ref{ic}) describes a
variety of dissipative problems, e.g., (degenerate) parabolic equations,
doubly-nonlinear equations, fractional and nonlocal problems, some ODEs, and
systems of reaction-diffusion equations \cite{Me}. Such a nonpotential
perturbation of doubly-nonlinear problems have been studied by many authors,
see, e.g., \cite{Ot1, Ot2} (see also \cite{Co, Co-Vi} for the potential case:
$f \equiv 0$).

Recently, a variational approach to the doubly-nonlinear system
(\ref{nonconvex target equation})-(\ref{ic}) has been proposed in
\cite{Ak-Me}. This approach relies on the so-called
\textit{Weighted-Energy-Dissipation} (WED) procedure for doubly nonlinear
systems \cite{Ak-Me, Ak-St, Ak-St2, Ak-St3, Mi-St}. Given a target
evolutionary problem, the WED approach consists in defining a global
parameter-dependent functional $I_{\varepsilon}$ over entire trajectories and
proving that its minimizers converge, up to subsequences, to solutions to the
target problem, as the parameter $\varepsilon$ goes to $0$.

The WED formalism has been used by Ilmanen \cite{Il} in the context of
mean-curvature flows, and later reconsidered by Mielke and Ortiz \cite{Mi-Or}
for rate-independent systems. The gradient flow case with $\lambda$-convex
potentials has been studied by Mielke and Stefanelli~\cite{Mi-St}. Akagi and
Stefanelli have extended the theory to the genuinely nonconvex case for
gradient flows \cite{Ak-St} and to convex doubly-nonlinear systems
\cite{Ak-St2, Ak-St3}, namely to problem (\ref{nonconvex target equation}%
)-(\ref{ic}) with $\varphi^{2}=0$ and $f=0$. Finally, an analogous approach
has been applied to some hyperbolic problems, e.g., the semilinear wave
equation \cite{Li-St,Se-Ti,St}, and to Lagrangian Mechanics equations
\cite{Li-St2}.

In the case of $f\neq0$, the lack of potential for $f$ opens on the one hand
the possibility of considering systems instead of equations. On the other
hand, it determines an obstruction to the application of the WED procedure
described above to problem (\ref{nonconvex target equation})-(\ref{ic}), for
the latter has in general no variational nature. In particular, it is not
possible to build a WED functional for problem
(\ref{nonconvex target equation})-(\ref{ic}). This difficulty may be tamed by
combining the WED technique with a fixed-point argument \cite{Ak-Me, Me}.
Since our argument relies on the WED procedure for problem
(\ref{nonconvex target equation})-(\ref{ic}), we now briefly sketch the
results in \cite{Ak-Me}, for the reader's convenience. Under the assumption of
Fr\'{e}chet differentiability of $\varphi^{2}$ and of $p$-growth for the
dissipation potential $\psi$, for all $v\in L^{p}(0,T;V)$ the WED-type
functional $I_{\varepsilon,w}:L^{p}(0,T;V)\rightarrow(-\infty,+\infty]$ is
defined by
\begin{align}
I_{\varepsilon,w}(u) &  =\left\{
\begin{array}
[c]{cc}%
\displaystyle{\int_{0}^{T}\mathrm{e}^{-t/\varepsilon}\left(  \varepsilon
\psi(u^{\prime})+\varphi^{1}(u)-\left\langle w,u\right\rangle _{V}\right)
\mathrm{d}t} & \text{if }u\in K(u_{0})\cap L^{m}(0,T;X)\text{,}\\
+\infty & \text{else,}%
\end{array}
\right.  \label{wed}\\
w &  =F(v):=\mathrm{d}_{V}\varphi^{2}(v)+f(v)\text{,}%
\end{align}
where $K(u_{0})=\{u\in W^{1,p}(0,T;V):u(0)=u_{0}\}$. For all $v\in
L^{p}(0,T;V)$, an approximation of the Direct Method \cite{Da} ensures that
the functional $I_{\varepsilon,w}$ admits a unique minimizer $u_{\varepsilon
,v}$ over $K(u_{0})$. Moreover, the map
\begin{equation}
S:v\in L^{p}(0,T;V)\longmapsto u_{\varepsilon,v}\in L^{p}(0,T;V)\label{esse}%
\end{equation}
can be proved to have a fixed-point $u_{\varepsilon}$ fulfilling
\begin{align}
u_{\varepsilon} &  =\underset{\tilde{u}}{\argmin}\text{ }I_{\varepsilon
,F(u_{\varepsilon})}(\tilde{u})\nonumber\\
&  =\underset{\tilde{u}}{\argmin}\text{ }\int_{0}^{T}\mathrm{e}%
^{-t/\varepsilon}\left(  \varepsilon\psi(\tilde{u}^{\prime})+\varphi
^{1}(\tilde{u})-\left\langle \mathrm{d}_{V}\varphi^{2}(u_{\varepsilon
})+f(u_{\varepsilon}),\tilde{u}\right\rangle _{V}\right)  \mathrm{d}%
t\label{fixed point wed}%
\end{align}
and solving an elliptic-in-time regularization of
(\ref{nonconvex target equation})-(\ref{ic}) given by
\begin{align}
-\varepsilon\left(  \mathrm{d}_{V}\psi(u_{\varepsilon}^{\prime})\right)
^{\prime}+\mathrm{d}_{V}\psi(u_{\varepsilon}^{\prime})+\eta^{1}-\eta
^{2}-f(u_{\varepsilon}) &  =0\text{ a.e. in }(0,T)\text{,}%
\label{nonconvex target equation eps}\\
\eta^{1} &  \in\partial\varphi^{1}(u_{\varepsilon})\text{, }\eta
^{2}=\mathrm{d}_{V}\varphi^{2}(u_{\varepsilon})\text{,}\\
\mathrm{d}_{V}\psi(u_{\varepsilon}^{\prime})(T) &  =0\text{,}\\
u_{\varepsilon}(0) &  =u_{0}\text{.}\label{ic eps}%
\end{align}
Finally, $u_{\varepsilon}$ converges, up to subsequences, to solutions to
(\ref{nonconvex target equation})-(\ref{ic}).

In the first part of this work we prove qualitative properties such as
symmetries, monotonicity, upper and lower bounds, sign properties, for
solutions to system (\ref{nonconvex target equation})-(\ref{ic}), provided
some compatibility conditions (e.g. symmetry of the domain or compatibility of
the initial data). More precisely, as we deal with equations with possibly
nonunique solutions, we prove the existence of at least one solution to
problem (\ref{nonconvex target equation})-(\ref{ic}) satisfying the
qualitative property. A standard approach suggests to describe qualitative
properties (e.g., the axial symmetry of a function) as invariance under the
action of a map \cite{Ca} (e.g. the reflection with respect to a given axis).
Following this idea we aim to prove existence of solutions $u$ to system
(\ref{nonconvex target equation})-(\ref{ic}) which are invariant under the
action of a map $R$, namely such that $u=Ru$. This will follow by i) proving
that the functional $I_{\varepsilon,w}$ is nonincreasing under the action of
the map $R$ and ii) checking that the invariance property is preserved by
taking the limit $\varepsilon\rightarrow0$.

Let us now briefly comment on some peculiarities and advantages of our method
and compare it with other techniques used to prove qualitative properties of
solutions to PDEs. We start by observing that our result is extremely
versatile. Indeed it applies to a large number of qualitative properties
(symmetries, upper and lower bounds, monotonicity, sign properties, and
combinations of them, see Corollary \ref{coroll composition}), and a variety
of evolution equations, e.g., dissipative systems of the form
(\ref{nonconvex target equation})-(\ref{ic}) (see Section 3), but also
rate-independent systems and hyperbolic problems (see Section
\ref{section other}).

Let us also note that our technique applies to maps $R$ which are not
necessary invertible (such as rearrangements or truncations). In particular,
$R$ does not generate a group of transformations. This implies that the theory
of invariance under the action of Lie groups (see, e.g., \cite{Ca}) may not be
directly used in our setting.

As a byproduct of our results, we get also existence of $R$-invariant
solutions to the elliptic-in-time regularization
(\ref{nonconvex target equation eps})-(\ref{ic eps}) of
(\ref{nonconvex target equation})-(\ref{ic}).

It is worth noting that our technique does not require regularity of solutions
to the target problem. This is not the case for others methods used for
proving qualitative properties of solutions to PDEs. Moving planes and sliding
methods \cite{Be-Ni, Br2} for instance require classical regularity, as they
rely on classical comparison principles and on the Hopf Lemma.

Furthermore, we can treat the case of problems with nonunique solutions.
Indeed, the uniqueness of solution to (\ref{nonconvex target equation}%
)-(\ref{ic}) may genuinely fail (e.g., the sublinear heat equation
$u_{t}-\Delta u=u^{q}$, $0<q<1$ has positive solutions even for zero initial
data). In this case it might be trivial to prove that $R$ maps solutions into
solutions (namely the problem is invariant under the action of $R$). However,
due to the lack of uniqueness one cannot conclude the existence of invariant
solutions. Our method is hence particularly useful in the case of
nonuniqueness of solutions.

In the second part of this work we use the WED approach to prove a comparison
principle for system (\ref{nonconvex target equation})-(\ref{ic}) in the case
of $f$ being independent of $u$. Our strategy consists in combining the WED
minimization with an abstract comparison principle, see Lemma
\ref{abstract comparison principle} below. More precisely, we i) prove a
comparison principle for minimizers of the WED functional and ii) pass to the
limit as $\varepsilon\rightarrow0$. It is noteworthy that the comparison
principle established in the present paper is not standard: given two initial
data $u_{0},v_{0}$ such that $u_{0}\leq v_{0}$ in a suitable sense, we show
the existence of at least two solutions $u,v$ such that $u(0)=u_{0}$,
$v(0)=v_{0}$, and $u\leq v$. We emphasize that we cannot expect the relation
$u\leq v$ to hold for all $u,v$ solutions to (\ref{nonconvex target equation}%
)-(\ref{ic}) such that $u(0)=u_{0}$, $v(0)=v_{0}$, as problem
(\ref{nonconvex target equation})-(\ref{ic}) has in general nonunique solutions.

Section \ref{section other} addresses by similar techniques different types of
evolution equations. In particular, we prove a comparison principle for
rate-independent systems of the form%
\[
\partial\psi(u^{\prime})+\partial\phi(u)-\Delta u\ni0\text{ in }\Omega
\times(0,T)\text{, }%
\]
where $\psi$ and $\phi$ are proper, lower semicontinuous, convex functionals,
$\psi$ is $1$-homogeneous, and $\Omega$ is a bounded subset of
$\mathbb{R}^{d}$.

Moreover, we check symmetries of solutions to the semilinear wave equation
\[
\rho u_{tt}+\nu u_{t}-\Delta u+F^{\prime}(u)=0\text{ in }\Omega\times(0,T),
\]
where $\Omega\subset\mathbb{R}^{d}$ is open, $\rho>0,~v\geq0$ are constants,
and $F\in C^{1}(\mathbb{R})$ has polynomial growth.

Finally, we tackle the lagrangian system
\[
Mu_{tt}+\nu u_{t}+\nabla U(u)=0\text{ in }(0,T)\text{, }%
\]
where $u:(0,T) \rightarrow\mathbb{R}^{d}$, $M$ is a positive definite
$d\times d$ matrix, $U\in C^{1}(\mathbb{R}^{d})$ is bounded from below and
convex, and $\nu\geq0$.

The paper is organized as follows. We fix the notation, enlist assumptions,
and we state and prove our abstract results in Section
\ref{section assumptions}. We present several examples of application to PDEs
and integrodifferential problems in Section \ref{section examples}. Finally,
Section \ref{section other} is devoted to rate-independent systems and
hyperbolic problems.

\section{Notation, assumptions, and main results \label{section assumptions}%
\label{section notation}}

Given any real Banach space $E$, we denote by $E^{\ast}$ its dual, by
$|\cdot|_{E}$ its norm, and by $\left\langle \cdot,\cdot\right\rangle _{E}$
the duality pairing between $E^{\ast}$ and $E$. Let $\phi:E\rightarrow
(-\infty,+\infty]$ be a convex functional, we denote its subdifferential by
$\partial_{E}\phi$ and its Fr\'{e}chet differential by $\mathrm{d}_{E}\phi$,
whenever it exists.

For all $h>1$, $\Theta_{h}(E)$ denotes the set of all lower semicontinuous
convex functionals $\phi:E\rightarrow\lbrack0,+\infty)$ such that there exists
a strictly positive constant $C$ such that
\begin{align*}
|u|_{E}^{h}  &  \leq C\left(  \phi(u)+1\right)  \text{ for all }u\in E,\\
|\xi|_{E^{\ast}}^{h^{\prime}}  &  \leq C(|u|_{E}^{h}+1), h^{\prime}=\frac
{h}{h-1}~\text{\ for all }\xi\in\partial_{E} \phi(u)\text{.}%
\end{align*}

Given a set $A$ and a map $R:A\rightarrow A$, we denote the set of fixed
points of $R$ by $A_{R}$, namely $A_{R}=\{a\in A:Ra=a\}$ is the set of
$R$\textit{-invariant} elements of $A$.

The symbols $\gamma^{+}$ and $\gamma^{-}$ stand for the positive and the
negative part in $\mathbb{R}$, namely $\gamma^{+}=\max\{\gamma,0\}$ and
$\gamma^{-}=-\min\{\gamma,0\}$, while the symbols $\vee$ and $\wedge$ denote
the maximum and the minimum respectively: $a\vee b=\max\{a,b\}$, $a\wedge
b=\min\{a,b\}$.

Let $V$ be a uniformly convex Banach space and $X$ be a reflexive Banach space
such that%
\[
X\hookrightarrow V\text{ and }V^{\ast}\hookrightarrow X^{\ast}%
\]
with densely-defined compact canonical injections. Let $\psi,\varphi
^{1},\varphi^{2}:V\rightarrow\lbrack0,\infty)$ be proper, lower semicontinuous
(l.s.c.), and convex functionals. Furthermore, we assume $\psi$ to be
Fr\'{e}chet differentiable and $\varphi^{1}$ to be strictly convex. Let
$p,m\in(1,\infty)$ be fixed. Assume that $\psi\in\Theta_{p}(V)$, $\varphi
^{1}\in\Theta_{m}(X)$. Moreover, we ask for\ constants $k\in\lbrack
0,1),~C_{1}>0$, and a nondecreasing function $\ell$ on $[0,+\infty)$ such
that
\begin{equation}
\varphi^{2}(u)\leq k\varphi^{1}(u)+C_{1} \label{hp phi2}%
\end{equation}
for all $u\in D(\varphi^{1})$ and
\begin{equation}
|\eta^{2}|_{V^{\ast}}^{p^{\prime}}\leq\ell(|u|_{V})(\varphi^{1}(u)+1)
\label{hp phi 2.2}%
\end{equation}
for all $u\in D(\varphi^{2})\ $and $\eta^{2}\in\partial_{V}\varphi^{2}(u)$.
Let $f:V\rightarrow V^{\ast}$ be such that
\begin{equation}
|f(u)|_{V^{\ast}}^{p^{\prime}}\leq C_{2}(|u|_{V}^{p}+1) \label{hp f}%
\end{equation}
for some constant $C_{2}\geq0$ and $f:L^{p}(0,T;V)\rightarrow L^{p^{\prime}%
}(0,T;V^{\ast})$ be continuous. Finally, we assume $u_{0}\in D(\varphi^{1})$.

Before stating our main results, let us now introduce the definition of strong
solution to system (\ref{nonconvex target equation})-(\ref{ic}).

\begin{definition}
[Strong solution]\label{solution to nonconvex target}A function
$u:[0,T]\rightarrow V$ is a \emph{strong solution}\ of system
\emph{(\ref{nonconvex target equation})-(\ref{ic}) }if

\begin{enumerate}
\item $u\in L^{m}(0,T;X)\cap W^{1,p}(0,T;V)$, $u(t)\in D(\partial_{V}%
\varphi^{1}(u))$ for a.e. $t\in(0,T)$, $\mathrm{d}_{V}\psi(u^{\prime})\in
L^{p^{\prime}}(0,T;V^{\ast})$,

\item there exist $\eta^{1}\in L^{m^{\prime}}(0,T;X^{\ast})\cap L^{p^{\prime}%
}(0,T;V^{\ast})$ and $\eta^{2}\in L^{p^{\prime}}(0,T;V^{\ast})$ such that
$\eta^{1}\in\partial_{V}\varphi^{1}(u)$ and $\eta^{2}\in\partial_{V}%
\varphi^{2}(u)$,

\item $\mathrm{d}_{V}\psi(u^{\prime})+\eta^{1}-\eta^{2}-f(u)=0$ in $V^{\ast}$
a.e. in $(0,T)$, and $u(0)=u_{0}$.
\end{enumerate}
\end{definition}

\subsection{Main result 1: qualitative properties\label{sec qual prop}}

In order to define a single-valued map $S$ as in (\ref{esse}), in this
subsection we additionally assume that $\varphi^{2}:V\rightarrow
(-\infty,+\infty)$ is Fr\'{e}chet differentiable.

We now introduce assumptions on the abstract maps $R:V\rightarrow V$ which
describe qualitative properties.

\begin{description}
\item[(R1)] $V_{R}$ is nonempty, convex, and closed in $V$. Assume that $Ru\in
W^{1,p}\left(  0,T;V\right)  $ for every $u\in W^{1,p}\left(  0,T;V\right)  $.

\item[(R2)] Define $F(v):=(\mathrm{d}_{V}\varphi^{2}+f)(v)$, and assume either
$\delta V_{R}\subset V_{R}$ for every $\delta\in(0,1)$ and
\begin{equation}
I_{\varepsilon,w}(Ru)\leq I_{\varepsilon,w}(u)\text{ for all }u\in
K(u_{0})\text{ and }w\in\{F(v):v\in L^{p}(0,T;V_{R})\} \label{1 approach}%
\end{equation}
or
\begin{equation}
I_{\varepsilon,w}(Ru)\leq I_{\varepsilon,w}(u)\text{ for all }u\in
K(u_{0})\text{ and }w=F(u). \label{2 approach}%
\end{equation}

\end{description}

Before commenting our assumptions let us state our main results.

\begin{theorem}
[Existence of invariant solutions]\label{abstract invariance principle}Let the
above assumptions be satisfied and $Ru_{0}=u_{0}$. Then, system
\emph{(\ref{nonconvex target equation})-(\ref{ic})} admits a strong solution
$u$ which is invariant under the action of $R$. Namely, $u=Ru.$
\end{theorem}

The latter result can be extended to composition of maps. More precisely, we
prove the following.

\begin{corollary}
[Composition of maps]\label{coroll composition}Let the assumptions of Theorem
\emph{\ref{abstract invariance principle}} be satisfied, let $R_{1}$ satisfy
\emph{(R1)} and \emph{(\ref{1 approach})}, and $R_{2}$ satisfy
\emph{(R1)-(R2)}. Moreover, assume $R_{1}u_{0}=R_{2}u_{0}=u_{0}$. Then, there
exists a strong solution $u$ to system \emph{(\ref{nonconvex target equation}%
)-(\ref{ic})} invariant under the action of both $R_{1}\circ R_{2}$ and
$R_{2}\circ R_{1}$.
\end{corollary}

We now comment briefly our abstract assumptions. Loosely speaking condition
(R1) ensures the compatibility of the map $R$ with the WED approach. More
precisely (R1), together with $Ru_{0}=u_{0}$ is sufficient to guarantee the
$R$-invariance of the domain of the WED functional, i.e., $RK(u_{0})\subset
K(u_{0})$. Assumption (R2) is the crucial assumption; it allows us to prove
that the map $S$ defined by (\ref{esse}) has a $R$-invariant fixed point
$u_{\varepsilon}$, i.e., $Ru_{\varepsilon}=u_{\varepsilon}$. Having this, by
using again (R1), we can easily pass to the limit $\varepsilon\to0$ and prove
$Ru=u= \lim_{\varepsilon\to0}u_{\varepsilon}$. Let us note that in concrete
applications (see Section 3) we check (R2) by proving the following.

\begin{description}
\item[(R2.1)] $\varphi^{1}(Ru)\leq\varphi^{1}(u)$ for all $u\in V$.

\item[(R2.2)] $\int_{0}^{T}\mathrm{e}^{-t/\varepsilon}\psi\left(
\frac{\mathrm{d}}{\mathrm{d}t}\left(  Ru\right)  \right)  \mathrm{d}t\leq
\int_{0}^{T}\mathrm{e}^{-t/\varepsilon}\psi\left(  \frac{\mathrm{d}%
}{\mathrm{d}t}u\right)  \mathrm{d}t$ for all $u\in W^{1,p}(0,T;V)$.

\item[(R2.3)] Either%
\begin{equation}
\left\langle w,Ru\right\rangle _{V}\geq\left\langle w,u\right\rangle
_{V}\text{ for all }w=F(v)\text{, }v\in V_{R}\text{, }u\in V\text{,}
\label{1 approach 2}%
\end{equation}
or
\begin{equation}
\left\langle F(u),Ru\right\rangle _{V}\geq\left\langle F(u),u\right\rangle
_{V}\text{ for all }u\in V. \label{2 approach 2}%
\end{equation}

\end{description}

\begin{lemma}
[(R2.1)-(R2.3) imply (R2)]Let \emph{(R1)} be satisfied. Then,

\begin{description}
\item[i)] \emph{(R2.1)}, \emph{(R2.2)}, and \emph{(\ref{1 approach 2})} imply
\emph{(\ref{1 approach})},

\item[ii)] \emph{(R2.1)}, \emph{(R2.2)}, and \emph{(\ref{2 approach 2})} imply
\emph{(\ref{2 approach}).}
\end{description}

In particular, \emph{(R2.1)-(R2.3)} imply \emph{(R2)}.
\end{lemma}

\begin{proof}
For all $u\in L^{p}(0,T;V)$ and $w\in L^{p^{\prime}}(0,T;V^{\ast})$, decompose
$I_{\varepsilon,w}(u)=I_{\varepsilon}^{1}(u)+I_{\varepsilon,w}^{2}(u)$, where
\begin{align*}
I_{\varepsilon}^{1}(u)  &  =\left\{
\begin{array}
[c]{cc}%
\displaystyle{\int_{0}^{T}\mathrm{e}^{-t/\varepsilon}\left(  \varepsilon
\psi(u^{\prime})+\varphi^{1}(u)\right)  \mathrm{d}t} & \text{if }u\in
K(u_{0})\cap L^{m}(0,T;X)\text{,}\\
+\infty & \text{else,}%
\end{array}
\right. \\
I_{\varepsilon,w}^{2}(u)  &  =-\int_{0}^{T}\mathrm{e}^{-t/\varepsilon
}\left\langle w,u\right\rangle _{V}\mathrm{d}t.
\end{align*}
Note that, as a consequence of $\varphi^{1}\in\Theta_{m}(X)$ and of (R2.1), we
have that $Ru\in X$ for all $u\in X$. This fact, (R2.1), and (R2.2) imply that
$I_{\varepsilon}^{1}(Ru)\leq I_{\varepsilon}^{1}(u)$ for all $u\in K(u_{0})$.
Moreover, inequality (\ref{1 approach 2}) ensures that $I_{\varepsilon,w}%
^{2}(Ru)\leq I_{\varepsilon,w}^{2}(u)$ for all $u\in K(u_{0})$ and
$w\in\{F(v):v\in L^{p}(0,T;V_{R})\}$, which yields inequality
(\ref{1 approach}), and inequality (\ref{2 approach 2}) implies
$I_{\varepsilon,w}^{2}(Ru)\leq I_{\varepsilon,w}^{2}(u)$ for all $u\in
K(u_{0})$ and $w=F(u)$, i.e. inequality (\ref{2 approach}).
\end{proof}

\subsubsection{Preliminary results for the proof of Theorem
\ref{abstract invariance principle}}

In order to prove Theorem \ref{abstract invariance principle}, we first
collect some preliminary results. We record here a slightly modified version of the Schaefer fixed-point
Theorem, which will be used in the proof of Theorem
\ref{abstract invariance principle}.

\begin{theorem}
[Modified Schaefer's fixed-point Theorem]\label{generalized Schaefer}Let $B$
be a reflexive Banach space and $L\subset B$ be nonempty, convex, and closed.
Assume $\delta L\subset L$ for every $\delta\in(0,1).$ Let $S:B\rightarrow B$
be continuous, compact, and such that $S(L)\subset L$. Moreover, let the set
$\{u\in B:\alpha S(u)=u$ for some $\alpha\in\lbrack0,1]\}$ be bounded. Then,
$S$ has a fixed point in $L$.
\end{theorem}

\begin{proof}
Our proof is a minor modification of the proof of the Schaefer fixed-point
Theorem presented in \cite[Thm. 4, Ch. 9]{Ev}. Choose $M$ so large that
$|u|_{B}<M$ for every $u\in\{u\in B:\alpha S(u)=u$ for some $\alpha\in
\lbrack0,1]\}$. Then, define
\[
T(u)=\left\{
\begin{array}
[c]{cc}%
S(u) & \text{if }|S(u)|_{B}\leq M\text{,}\\
\frac{S(u)M}{|S(u)|_{B}} & \text{if }|S(u)|_{B}>M\text{.}%
\end{array}
\right.
\]
Observe that $T:B_{M}(0)\cap L\rightarrow B_{M}(0)\cap L$, where
$B_{M}(0)=\{b\in B:|b|_{B}\leq M\}$. Define $\tilde{K}$ to be the convex hull
of $T(B_{M}(0)\cap L)$ and $K$ the closure of $\tilde{K}$.\ Note that, as $L$
is closed and convex, we have that $K\subseteq B_{M}(0)\cap L$. Moreover,
$T:K\rightarrow K$ is continuous and $T(K)$ is relatively compact in $K$.
Hence, we can apply the Schauder fixed-point Theorem and prove the existence
of $u\in K\subset L$ such that $T(u)=u$. We now show that $u$ is a fixed point
for $S$. Suppose by contradiction that $S(u)\neq u$. Then, $|S(u)|_{B}>M$ and
$u=\alpha S\left(  u\right)  $ for $\alpha=\frac{M}{|S(u)|_{B}}<1$. Hence,
$|u|_{B}<M$. As $u$ is a fixed point for $T$, we conclude that $|T(u)|_{B}%
=|u|_{B}<M=\left\vert \frac{MS(u)}{|S(u)|_{B}}\right\vert _{B}=|T(u)|_{B}$, a contradiction.
\end{proof}

We shall now summarize the WED approach to system
(\ref{nonconvex target equation})-(\ref{ic}) studied in \cite{Ak-Me}.

\begin{proposition}
[WED approach I]\label{wed approach nonpotential}Let the assumptions of
Theorem \emph{\ref{abstract invariance principle}} be satisfied. Then, the
functional $I_{\varepsilon,w}$ defined by%
\begin{equation}
I_{\varepsilon,w}(u)=\left\{
\begin{array}
[c]{cc}%
\displaystyle{\int_{0}^{T}\mathrm{e}^{-t/\varepsilon}\left(  \varepsilon
\psi(u^{\prime})+\varphi^{1}(u)-\left\langle w,u\right\rangle _{V}\right)
\mathrm{d}t} & \text{if }u\in K(u_{0})\cap L^{m}(0,T;X)\text{,}\\
+\infty & \text{else,}%
\end{array}
\right.  \label{wed functional dn}%
\end{equation}
admits an unique minimizer $u_{\varepsilon,w}$ over the set $K(u_{0})=\{u\in
W^{1,p}\left(  0,T;V\right)  :u(0)=u_{0}\}$ for every $w\in L^{p^{\prime}%
}(0,T;V^{\ast})$ and $\varepsilon>0$ small enough.

Define the map $S:L^{p}(0,T;V)\rightarrow L^{p}(0,T;V)$ by
\begin{equation}
S:v\longmapsto w=F(v):=f(v)+\mathrm{d}_{V}\varphi^{2}(v)\longmapsto
u_{\varepsilon,v}, \label{def S}%
\end{equation}
where
\[
u_{\varepsilon,v}=\underset{\tilde{u}\in K(u_{0})}{\argmin}\text{
}I_{\varepsilon,w}(\tilde{u}).
\]
Then, $S$ is continuous and compact, the set $\{v\in L^{p}(0,T;V):\alpha
S(v)=v$ for $\alpha\in\lbrack0,1]\}$ is bounded in $L^{p}(0,T;V)$, and $S$ has
a fixed\ point $u_{\varepsilon}$ for every $\varepsilon>0$ small enough.
Moreover, $u_{\varepsilon}\in C([0,T];V)$ and fulfills
\begin{equation}
u_{\varepsilon}=\underset{\tilde{u}\in K(u_{0})}{\argmin}\text{ }%
I_{\varepsilon,F(u_{\varepsilon})}(\tilde{u})=\underset{\tilde{u}\in K(u_{0}%
)}{\argmin}\text{ }\int_{0}^{T}\mathrm{e}^{-t/\varepsilon}(\varepsilon
\psi(\tilde{u}^{\prime})+\varphi^{1}(\tilde{u})-\left\langle F(u_{\varepsilon
}),\tilde{u}\right\rangle _{V})\mathrm{d}t. \label{Formula}%
\end{equation}
Finally, there exists a sequence $\varepsilon_{n}\rightarrow0$ such that
\[
u_{\varepsilon_{n}}\rightarrow u\text{ strongly in }C\left(  [0,T];V\right)
,
\]
where $u$ is a strong solution of \emph{(\ref{nonconvex target equation}%
)-(\ref{ic}).}
\end{proposition}

\subsubsection{Proof of Theorem \ref{abstract invariance principle} and
Corollary \ref{coroll composition}}

Let us first prove Theorem \ref{abstract invariance principle}. We start by
checking that the map $S:L^{p}(0,T;V)\rightarrow L^{p}(0,T;V)$ defined as in
(\ref{def S}) has a fixed point $u_{\varepsilon}\in C([0,T];V)$ such that
\begin{equation}
Ru_{\varepsilon}=u_{\varepsilon}\text{ in }[0,T]\text{. }
\label{invariance eps}%
\end{equation}
In case condition (\ref{2 approach}) is satisfied we proceed as follows. Every fixed
point $u_{\varepsilon}$ of $S$ fulfills (see Theorem
\ref{wed approach nonpotential})%
\[
u_{\varepsilon}=\arg\min_{\tilde{u}\in K(u_{0})}I_{\varepsilon
,F(u_{\varepsilon})}(\tilde{u})\text{,}%
\]
i.e., $u_{\varepsilon}$ is the unique minimizer of $I_{\varepsilon
,F(u_{\varepsilon})}$ over $K(u_{0})$. As a consequence of (R1), we have that
$Ru_{\varepsilon}\in K(u_{0})$, and, thanks to assumption (\ref{2 approach}),%
\[
I_{\varepsilon,F(u_{\varepsilon})}(Ru_{\varepsilon})\leq I_{\varepsilon
,F(u_{\varepsilon})}(u_{\varepsilon})\text{.}%
\]
By uniqueness of the minimizer, $Ru_{\varepsilon}=u_{\varepsilon}$.

In case condition (\ref{1 approach}) holds true, we show that $S$ maps the set
$L^{p}(0,T;V_{R})$ into itself and we than infer the existence of a fixed
point $u_{\varepsilon}\in L^{p}(0,T;V_{R})$ of $S$. Let $v\in L^{p}%
(0,T;V_{R})$ and $w=F(v)$. Using assumptions (R1) and (\ref{1 approach}) and arguing as
above, we deduce that the unique minimizer $u_{\varepsilon,v}$ of
$I_{\varepsilon,w}$ satisfies $Ru_{\varepsilon,v}=u_{\varepsilon,v}$. This
yields $S:L^{p}(0,T;V_{R})\rightarrow L^{p}(0,T;V_{R})$. By virtue of Theorem
\ref{wed approach nonpotential}, the map $S:L^{p}(0,T;V_{R})\rightarrow
L^{p}(0,T;V_{R})$ is continuous, compact, and such that the set $\{v\in
L^{p}(0,T;V_{R}):\alpha S(v)=v$ for $\alpha\in\lbrack0,1]\}$ is bounded.
Moreover, we have that $\delta L^{p}(0,T;V_{R}%
)\subset L^{p}(0,T;V_{R})$ for all $\delta\in(0,1)$. Therefore, by applying
the Schaefer fixed-point Theorem \ref{generalized Schaefer}, we conclude that
$S$ has a fixed point $u_{\varepsilon}\in L^{p}(0,T;V_{R})$. Moreover,
$u_{\varepsilon}\in C\left(  [0,T];V_{R}\right)  $. In particular, it fulfills
(\ref{invariance eps}).

Thanks to Theorem \ref{wed approach nonpotential}, we have (after extraction
of a not relabeled subsequence) that $u_{\varepsilon}\rightarrow u$ strongly
in $C\left(  [0,T];V\right)$, where $u$ solves
(\ref{nonconvex target equation})-(\ref{ic}). As $V_{R}$ is closed in $V$,
from (\ref{invariance eps}) we deduce that%
\[
Ru(t)=u(t)\text{ for all }t\in\lbrack0,T]\text{.}%
\]
This proves Theorem \ref{abstract invariance principle}.

We now move to Corollary \ref{coroll composition}. Assume that $R_{1}$ and
$R_{2}$ satisfy (\ref{1 approach}). Then, by restricting the map $S$ to
$L^{p}(0,T;V_{R_{1}}\cap V_{R_{2}})$ and arguing as above, we can easily
deduce that $S:L^{p}(0,T;V_{R_{1}}\cap V_{R_{2}})\rightarrow L^{p}%
(0,T;V_{R_{1}}\cap V_{R_{2}})$ and that it has a fixed point $u_{\varepsilon
}\in L^{p}(0,T;V_{R_{1}}\cap V_{R_{2}})$.

In case $R_{1}$ satisfies condition (\ref{1 approach}) and $R_{2}$ satisfies
condition (\ref{2 approach}), we still have that $S:L^{p}(0,T,V_{R_{1}%
})\rightarrow L^{p}(0,T,V_{R_{1}})$ and it has a fixed point $u_{\varepsilon
}\in L^{p}(0,T,V_{R_{1}})$. Moreover, as a consequence of assumption
(\ref{2 approach}) for $R_{2}$%
\[
I_{\varepsilon,F(u_{\varepsilon})}(R_{2}u_{\varepsilon})\leq I_{\varepsilon
,F(u_{\varepsilon})}(u_{\varepsilon})\text{.}%
\]
Thus, $R_{2}u_{\varepsilon}=u_{\varepsilon}$ in $[0,T]$, which yields
$u_{\varepsilon}\in L^{p}(0,T;V_{R_{1}}\cap V_{R_{2}})$.

By applying Theorem \ref{wed approach nonpotential}, we can pass to the limit
as $\varepsilon\rightarrow0$ and obtain that there exists $u$ solution of
system (\ref{nonconvex target equation})-(\ref{ic}) such that
\[
u\in C([0,T];V_{R_{1}}\cap V_{R_{2}})\text{.}%
\]
In particular,%
\[
R_{2}R_{1}u=R_{1}R_{2}u=u\text{ in }[0,T]\text{.}%
\]
This concludes the proof of Corollary \ref{coroll composition}.

Note that, as a byproduct of Theorem \ref{abstract comparison principle} and
Corollary \ref{coroll composition}, we have existence of $R$-invariant
solutions to the elliptic-in-time regularization
(\ref{nonconvex target equation eps})-(\ref{ic eps}) of
(\ref{nonconvex target equation})-(\ref{ic}).

\subsection{Main result 2: comparison principles\label{section comparison}}

We now state a comparison principle for doubly-nonlinear systems. Here, we
assume $F$\ to have a potential structure, namely $F(u)=g+\partial_{V}%
\varphi^{2}(u)$: indeed comparison principles cannot be expected for genuinely
nonpotential terms. Counterexamples can be found already in ODE systems. On
the other hand, we allow for possibly noncontinuous/nondifferentiable
functionals $\varphi^{2}$ as our argument does not require uniqueness of the
minimizer of the WED functional.

Let $\psi$ and $\varphi^{1}$ satisfy assumptions of Theorem
\ref{abstract invariance principle}. Let $\varphi^{2}:V\rightarrow
\lbrack0,\infty)$ be proper, convex, l.s.c., and satisfy conditions
(\ref{hp phi2}) and (\ref{hp phi 2.2}). Assume additionally that the reaction
term $f$ does not depend on $u$, namely
\[
f(u)=g\in L^{p^{\prime}}(0,T;V^{\ast})\text{ for all }u\in L^{p}(0,T;V).
\]
In order to avoid unnecessary complications deriving by the definition of an
abstract concept of order in Banach spaces, we restrict now our attention to
the case of problems (\ref{nonconvex target equation})-(\ref{ic}) whose
solutions can be represented as real-valued functions. More precisely, we
assume that$\ X$ and $V$ are Banach spaces composed by real-valued functions
satisfying assumptions of Theorem \ref{abstract invariance principle} and such
that
\begin{equation}
\{(a,b)\in V\times V:a\leq b\}\text{ is closed in }V\times V
\label{clos cond 2}%
\end{equation}
and
\begin{align}
u\wedge v,u\vee v \in W^{1,p}(0,T;V)\text{ for all }u,v\in W^{1,p}(0,T;V).
\label{min w1p}%
\end{align}
Let us remark that the above assumptions are satisfied by the Lebesgue spaces
$L^{q}(\Omega)$, Sobolev spaces $W^{1,q}(\Omega)$, and fractional Sobolev
spaces $W^{s,r}(\Omega)$, for all $q\in\lbrack1,\infty]$, $s\in(0,1)$,
$r\in\lbrack1,\infty)$ and $\Omega=\mathbb{R}^{d}$ or measurable, bounded, and
with Lipschitz boundary. Define the WED functional
\begin{equation}
I_{\varepsilon}(\tilde{u})=\left\{
\begin{array}
[c]{cc}%
\begin{array}
[c]{cc}%
\displaystyle{\int_{0}^{T}\mathrm{e}^{-t/\varepsilon}\left(  \varepsilon
\psi(\tilde{u}^{\prime})+\varphi^{1}(\tilde{u})-\varphi^{2}(\tilde
{u})-\left\langle g,\tilde{u}\right\rangle _{V}\right)  \mathrm{d}t} & \\
\text{ } &
\end{array}
&
\begin{array}
[c]{c}%
\text{if }\tilde{u}\in K(u_{0})\\
\cap L^{m}(0,T;X)\text{,}%
\end{array}
\\
+\infty & \text{else,}%
\end{array}
\right.  \label{wed comparison}%
\end{equation}
where $K(u_{0})=\{\tilde{u}\in W^{1,p}(0,T;V):\tilde{u}(0)=u_{0}\}$, and
assume that for all $u_{0},v_{0}\in D(\varphi^{1})$, such that $u_{0} \leq
v_{0}$%
\begin{equation}
I_{\varepsilon}(u\wedge v)+I_{\varepsilon}(u\vee v)\leq I_{\varepsilon
}(u)+I_{\varepsilon}(v)\text{ for all }u\in K(u_{0})\text{, }v\in
K(v_{0})\text{.} \label{key assumption comparison}%
\end{equation}
Before stating the main result of this section let us remark that assumption
(\ref{key assumption comparison}) is crucial as it allows us to prove a
comparison principle for minimizers of the WED functional by applying Lemma
\ref{abstract comparison principle} below.

The main result of this section states a comparison principle for problem
(\ref{nonconvex target equation})-(\ref{ic}).

\begin{theorem}
[Comparison principle]\label{comp princ 2}Let $u_{0},v_{0}\in X$ be such that
$u_{0}\leq v_{0}$. Assume condition \emph{(\ref{clos cond 2}%
)-(\ref{key assumption comparison})}. Then, there exist two strong solutions
$u,v$ to equation \emph{(\ref{nonconvex target equation})-(\ref{ic})}
corresponding to the initial data $u_{0}$ and $v_{0}$, respectively, such that
$u\leq v$ for all $t\in\lbrack0,T]$.
\end{theorem}

Note that solutions to (\ref{nonconvex target equation})-(\ref{ic}) are, in
general, nonunique. Thus, we can not expect the statement of the theorem to
hold for every couple $u,v$ of solutions corresponding to the initial data
$u_{0}$ and $v_{0}$ (take $u_{0}=v_{0}$).

Several applications of Theorem \ref{comp princ 2} to local and nonlocal PDE
problems will be presented in Section \ref{section examples}.

\subsubsection{Preliminary results for the proof of Theorem \ref{comp princ 2}%
}

In this section we collect some preliminary results, which will be used in the
proof of Theorem \ref{comp princ 2}. Under the assumptions of Theorem
\ref{comp princ 2}, namely $f$ independent of $u$ the WED procedure simplifies
as a fixed-point argument is no longer necessary. More precisely, the
following proposition has been proved in \cite{Ak-Me} (see also \cite{Ak-St,
Ak-St2, Ak-St3, Mi-St}).

\begin{proposition}
[WED approach 2]\label{comp princ dn}Let the assumptions of Theorem
\emph{\ref{comp princ 2}} be satisfied. Then, for each $g\in L^{p^{\prime}%
}(0,T;V^{\ast})$ and $u_{0}\in X$ the \emph{WED} functional $I_{\varepsilon}$,
defined by \emph{(\ref{wed comparison}),} admits at least one global minimizer
$u_{\varepsilon}$ over the set $K(u_{0})=\{u\in W^{1,p}(0,T;V):u(0)=u_{0}\}$.
Moreover, for every sequence $\varepsilon_{n}\rightarrow0$ there exists a (not
relabeled) subsequence such that
\begin{equation}
u_{\varepsilon_{n}}\rightarrow u~\text{strongly in }C\left(  [0,T];V\right)
\label{strong conv 2}%
\end{equation}
and $u$ is a strong solution of system \emph{(\ref{nonconvex target equation}%
)-(\ref{ic})}.
\end{proposition}

In order to prove Theorem \ref{comp princ 2}, we take advantage of the
following abstract comparison principle for minimizers of functionals.

\begin{lemma}
[Abstract comparison principle]\label{abstract comparison principle}Let $A,B$
be sets. Let $\alpha,\beta:A\times A\rightarrow A$ be two maps. Let
$M_{0}:A\rightarrow B$ be a function. Let $I:A\rightarrow\mathbb{R}%
\cup\{+\infty\}$ be such that for every $\bar{u}\in B$ there exists at least a
minimizer of $I$ over the set $K(\bar{u})=\{w\in A:M_{0}(w)=\bar{u}\}$.
Assume
\begin{equation}
I(\alpha(u,v))+I(\beta(u,v))\leq I(u)+I(v)\text{,}
\label{condition for comparison princ}%
\end{equation}
for all $u\in K(u_{0})$, $v\in K(v_{0})$, $M_{0}\left(  \alpha(u,v)\right)
=v_{0}$, and $M_{0}\left(  \beta(u,v)\right)  =u_{0}$. Fix%
\begin{align*}
u  &  \in\arg\min_{w\in K(u_{0})}I(w)\text{,}\\
v  &  \in\arg\min_{w\in K(v_{0})}I(w)\text{.}%
\end{align*}
Then, $\alpha(u,v)$ and $\beta(u,v)$ are minimizers of $I$ over $K(v_{0})$ and
$K(u_{0})$ respectively. Furthermore, if additionally $I$ has a unique
minimizer over $K(\bar{u})$ for all $\bar{u}\in B$, then $\alpha(u,v)=v$ and
$\beta(u,v)=u$.
\end{lemma}

\begin{proof}
Let $u$ and $v$ be minimizers of $I$ over $K(u_{0})$ and $K(v_{0})$
respectively and let $M_{0}\left(  \beta(u,v)\right)  =u_{0}$ and
$M_{0}\left(  \alpha(u,v)\right)  =v_{0}$. Then, we have
\begin{align*}
I\left(  u\right)   &  \leq I\left(  \beta(u,v)\right)  \text{, }\\
I\left(  v\right)   &  \leq I\left(  \alpha(u,v)\right)  \text{.}%
\end{align*}
By using the property (\ref{condition for comparison princ}), we get
\begin{align*}
I\left(  u\right)   &  \leq I\left(  \beta(u,v)\right)  \leq
I(u)+I(v)-I(\alpha(u,v)),\\
I\left(  v\right)   &  \leq I\left(  \alpha(u,v)\right)  \leq
I(u)+I(v)-I(\beta(u,v)),
\end{align*}
Thus,%
\begin{align*}
I(\beta(u,v))  &  \leq I(u)\text{,}\\
I(\alpha(u,v))  &  \leq I(v)\text{ }%
\end{align*}
Therefore, $\alpha(u,v)$ minimizes $I$ over $K(v_{0})$ and $\beta(u,v)$
minimizes $I$ over $K(u_{0})$. If additionally the minimizers are unique, then
$v=\alpha(u,v)$ and $\beta(u,v)=u$.
\end{proof}

\subsubsection{Proof of Theorem \ref{comp princ 2}}

With this preparation we are now in the position of proving Theorem
\ref{comp princ 2}. Let $u_{\varepsilon}$ and $v_{\varepsilon}$ be minimizers
of $I_{\varepsilon}$ over $K(u_{0})$ and $K(v_{0})$ respectively. Recalling
that $K(\bar{u})=\{\tilde{u}\in W^{1,p}(0,T;V):\tilde{u}(0)=\bar{u}\}$, and
using assumptions (\ref{min w1p}), we have that $u_{\varepsilon}%
^{1}:=u_{\varepsilon}\wedge v_{\varepsilon}\in K(u_{0})$ and $v_{\varepsilon
}^{1}:=u_{\varepsilon}\vee v_{\varepsilon}\in K(v_{0})$. By applying Lemma
\ref{abstract comparison principle} with $A=W^{1,p}(0,T;V)$, $B=V$,
$\alpha(u,v)=u\vee v$, $\beta(u,v)=u\wedge v$, $M_{0}(w)=w(0)$, we get that
$u_{\varepsilon}^{1}$ and $v_{\varepsilon}^{1}$ minimize $I_{\varepsilon}$
over $K(u_{0})$ and $K(v_{0})$ respectively and $u_{\varepsilon}^{1}\leq
v_{\varepsilon}^{1}$ a.e. in $[0,T]$. Thanks to Proposition
\ref{comp princ dn}, there exists a sequence $\varepsilon_{n}\rightarrow0$
such that $u_{\varepsilon_{n}}^{1}\rightarrow u$ a.e. in $[0,T]$ and $u$ is a
strong solution to the doubly-nonlinear problem
(\ref{nonconvex target equation})-(\ref{ic}). Moreover, there exists a (not
relabeled) subsequence such that $v_{\varepsilon_{n}}^{1}\rightarrow v$ a.e.
in $[0,T]$ and $v$ solves system (\ref{nonconvex target equation})-(\ref{ic}).
Thus, thanks to convergence (\ref{strong conv 2}) and the closedness condition
(\ref{clos cond 2}), $u\leq v$.

Let us remark that the uniqueness of minimizers of the WED functional was not
used here.

\section{Applications\label{section examples}}

In this section we present several applications of Theorem
\ref{abstract invariance principle} and Theorem \ref{comp princ 2} to some PDE
problems of local and nonlocal type.

\subsection{Doubly-nonlinear parabolic equations}

Consider the family of \textit{doubly-nonlinear equations of }$m$%
\textit{-Laplace type} given by
\begin{align}
&  \alpha(u_{t})-\operatorname{div}(B(x)|\nabla u|^{m-2}\nabla
u)+C(x)|u|^{m-2}u\nonumber\\
&  -D(x)|u|^{q-2}u-h(x,t)=0\text{ in }\Omega\times(0,T)\text{,}%
\label{example 1}\\
&  u+b\frac{\partial u}{\partial n}=0\text{ on }\partial\Omega\times
(0,T)\text{,}\label{bc example 1}\\
&  u(0)=u_{0}\text{ in }\Omega. \label{ic example 1}%
\end{align}
Here, we assume that $\Omega\subset\mathbb{R}^{d}$ is bounded with Lipschitz
boundary $\partial\Omega$ and $\alpha:\mathbb{R}\rightarrow\mathbb{R}$ is
maximal monotone. Moreover, let exist a constant $\tilde{C}$ such that
\begin{equation}
\frac{1}{\tilde{C}}|s|^{p}-\tilde{C}\leq A(s):=\int_{0}^{s}\alpha
(r)\mathrm{d}r\text{ and }|\alpha(s)|^{p^{\prime}}\leq\tilde{C}\left(
|s|^{p}+1\right)  \text{ for all }s\in\mathbb{R}\text{.} \label{def alpha}%
\end{equation}
We assume $m,q,p$ to satisfy the following relations: $m \geq2 $,
$1<p<m^{\ast}:=dm/(d-m)^{+}$, $1<q\leq p$. We consider $b$ constant and
strictly positive. We remark that this choice is made for sake of simplicity
and other types of boundary conditions, e.g, Neumann or Dirichlet boundary
conditions can be treated similarly and with no additional difficulties. Let
$h\in L^{p^{\prime}}(0,T;L^{p^{\prime}}(\Omega))$. We assume the coefficients
$B,C,D\in L^{\infty}(\Omega)$ to be positive a.e. in $\Omega$. Moreover,
$0<b_{1}\leq B(x)$ for a.e. $x\in\Omega$ and some $b_{1} \in\mathbb{R}$.

With the aim of applying the abstract theory of Section \ref{sec qual prop},
we recast system (\ref{example 1})-(\ref{ic example 1}) into the abstract form
(\ref{nonconvex target equation})-(\ref{ic}). To this end, we set
$V=L^{p}(\Omega)$, $X=W^{1,m}(\Omega)$,\ and%
\begin{align}
\psi(u)  &  =\int_{\Omega}A(u)\text{,}\\
\varphi^{1}(u)  &  =\left\{
\begin{array}
[c]{ll}%
\displaystyle{\int_{\Omega}\left(  \frac{1}{m}B|\nabla u|^{m}+\frac{1}%
{m}C|u|^{m}\right)  \mathrm{d}x+\int_{\partial\Omega}\frac{1}{2b}%
|u|^{2}\text{,}} & \text{if }u\in W^{1,m}(\Omega)\\
+\infty & \text{otherwise,}%
\end{array}
\right. \\
f(u)  &  =0\text{, \ \ \ \ }\varphi^{2}(u)=D\frac{1}{q}|u|^{q}+\left\langle
h,u\right\rangle _{V}\text{.} \label{choice 2}%
\end{align}

System (\ref{example 1})-(\ref{ic example 1}) is a doubly-nonlinear version of
the Allen-Cahn equation coupled with Robin boundary conditions. The existence
of a strong solution $u$ to (\ref{example 1})-(\ref{ic example 1}) in the
sense of Definition \ref{solution to nonconvex target} follows by a direct
application of Proposition \ref{wed approach nonpotential} (for checking that
assumptions of Proposition \ref{wed approach nonpotential} are satisfied we
refer the reader to \cite{Ak-St, Ak-St2}). We recall that, if $u$ solves
(\ref{example 1})-(\ref{ic example 1}) in the sense of Definition
\ref{solution to nonconvex target}, then, $u\in W^{1,p}(0,T;L^{p}(\Omega))\cap
L^{m}(0,T;W^{1,m}(\Omega))$, $\operatorname{div}(B|\nabla u(t)|^{m-2}\nabla u(t))\in
L^{p^{\prime}}(\Omega)$ for a.e. $t\in(0,T)$, and $u$ solves (\ref{example 1})
pointwise a.e. in $\Omega\times(0,T)$. It is worth mentioning that the
uniqueness of solution may essentially fail, e.g., in the case of the
sublinear heat equation $u_{t}-\Delta u=|u|^{q-2}u$, $1<q<2$. Indeed, the
latter admits positive solutions even for zero initial data.

We aim at proving existence of solutions to (\ref{example 1}%
)-(\ref{ic example 1}) which satisfy qualitative properties such as
symmetries, monotonicity, and upper and lower bounds. To this end, we
introduce some maps $R:L^{p}\left(  \Omega\right)  \rightarrow L^{p}(\Omega)$
to describe the mentioned properties, together with compatibility assumptions
on the data.

\begin{enumerate}
\item \textit{Linear rigid transformation of the space:} $Ru(x)=u(rx)$ for
some $r\in SL(d,\mathbb{R})= \{ r \in M(\mathbb{R}^{d\times d}):|\det r|=1
\}$, $r\Omega=\Omega$, and $B,C,D,h$ are $R$-invariant;

\item \textit{Symmetric decreasing rearrangement} or \textit{Schwartz
symmetrization }\cite{Ka}\textit{: }$Ru=\left(  u^{+}\right)  ^{\ast}$,
$\Omega$ is radially symmetric, $B$, $C$, and $D$ are constant a.e. in
$\Omega$, and $h=\left(  h^{+}\right)  ^{\ast}$ a.e. in $\Omega\times(0,T)$;

\item \textit{Symmetric decreasing rearrangement with respect to the
hyperplane} $H\subset\mathbb{R}^{d}$ \textit{(}or\textit{ Steiner
symmetrization }in case $\dim H=1$) \cite{Ka}: $Ru=\left(  u^{+}\right)
^{\ast,H}$, $\Omega$ is invariant under the action of any rotation and
reflection which maps $H$ into $H$, $B$, $C$, and $D$ are constant a.e. in
$\Omega$, $h=\left(  h^{+}\right)  ^{\ast,H}$ a.e. in $\Omega\times(0,T)$;

\item \textit{Monotone decreasing rearrangement with respect to the direction
}$y\in\mathbb{R}^{d}$ \cite{Ka}: $Ru=\left(  u^{+}\right)  ^{\ast,y}$,
$\Omega=\Omega^{\ast, y}$, $B$, $C$, and $D$ are constant in the direction of
$y$, a.e. in $\Omega$, and $h=\left(  h^{+}\right)  ^{\ast,y}$ a.e. in
$\Omega\times(0,T)$;

\item \textit{Lower truncation}: $R(u)=(u-M)^{+}+M$, where $M\leq0$ is
constant, $h\geq0$ a.e. in $\Omega\times(0,T)$, and either $M=0$ or $D=0$;

\item \textit{Upper truncation:} $R(u)=M-(M-u)^{+}$, where $M\geq0$ is
constant, $h\leq0$ a.e. in $\Omega\times(0,T)$, and either $M=0$ or $D=0$.
\end{enumerate}

The definitions and some basic properties of the rearrangement maps are
collected in the appendix for the reader's convenience (see also \cite{Ka} for
a survey).

Assume $Ru_{0}=u_{0}$. By applying Theorem \ref{abstract invariance principle}%
, we conclude that there exists a solution $u$ to the Cauchy problem
(\ref{example 1})-(\ref{ic example 1}) such that $Ru=u$ a.e. in $\Omega
\times\lbrack0,T]$. Furthermore, by applying Theorem \ref{comp princ 2}, we
can also deduce a comparison principle for solutions to (\ref{example 1}%
)-(\ref{ic example 1}). More\ precisely, we have the following theorem.

\begin{theorem}
[Doubly-nonlinear parabolic equation (\ref{example 1})-(\ref{ic example 1}%
)]\label{thm example 1}\label{thm example 5}Let the above assumptions be
satisfied and let $R_{i}$, $i=1,...,k$, be any collection of maps as defined
above. Assume $R_{i}u_{0}=u_{0}$ for $i=1,...,k$. Then, there exists a strong
solution $u$ (in the sense of Definition
\emph{\ref{solution to nonconvex target}}) to the Cauchy problem
\emph{(\ref{example 1})-(\ref{ic example 1})} such that $R_{1}\circ...\circ
R_{k}u=u$.

Furthermore, let $u_{0}\leq v_{0}$ a.e. in $\Omega$. Then, there exist $u,v$
strong solutions to \emph{(\ref{example 1})-(\ref{bc example 1})} such that
$u(0)=u_{0}$, $v\left(  0\right)  =v_{0}$, and $u\leq v$ a.e. in $\Omega
\times(0,T)$.
\end{theorem}

Note that any map $R_{i}$ is associated with some compatibility conditions on
the data. Let us note that, in case $k\geq2$, these conditions have to be
satisfied simultaneously and, hence, they have to be compatible. This fact is
implicitly guaranteed by the assumption of the existence of some $u_{0}$
satisfying $u_{0}=R_{i}u_{0}$ for all $i=1,...,k$.

\begin{proof}
In order to apply Theorem \ref{abstract invariance principle}, it suffices to
check conditions (R1)-(R2). Note that $R$ satisfies
\begin{equation}
\int_{\Omega}J(Ru-Rv)\leq\int_{\Omega}J(u-v) \label{nonexpansivity}%
\end{equation}
for every $u,v\in V$ and $J:\mathbb{R}\rightarrow\mathbb{R}$ nonnegative,
convex, and such that $J(0)=0$ (see the appendix or \cite{Ka} for the case of
rearrangements). Thus,
\[
\lim_{s\rightarrow t}\int_{\Omega}\left\vert \frac{Ru(s)-Ru(t)}{s-t}%
\right\vert ^{p}\mathrm{d}x\leq\lim_{s\rightarrow t}\int_{\Omega}\left\vert
\frac{u(s)-u(t)}{s-t}\right\vert ^{p}\mathrm{d}x\text{ for a.a. }t\in(0,T).
\]
This fact, together with the Dominated Convergence Theorem, proves that
\[
Ru\in W^{1,p}(0,T;V)\text{ for all }u\in W^{1,p}(0,T;V).
\]
This easily yields (R1). It is standard matter to check (R2.1) (see the
appendix or \cite{Ka} for more details in the case of rearrangement maps). By
definition of $\psi$, for a.a. $t\in(0,T)$, we have%
\[
\psi\left(  \frac{\mathrm{d}}{\mathrm{d}t}(Ru(t))\right)  =\int_{\Omega
}A\left(  \lim_{s\rightarrow t}\frac{Ru(s)-Ru(t)}{s-t}\right)  \mathrm{d}%
x\text{.}%
\]
By using the continuity of $A$,
\begin{align*}
\psi\left(  \frac{\mathrm{d}}{\mathrm{d}t}(Ru(t))\right)   &  =\int_{\Omega
}\lim_{s\rightarrow t}A\left(  \frac{Ru(s)-Ru(t)}{s-t}\right)  \mathrm{d}x\\
&  =\lim_{s\rightarrow t}\int_{\Omega}A\left(  \frac{Ru(s)-Ru(t)}{s-t}\right)
\mathrm{d}x\text{.}%
\end{align*}
Thanks to inequality (\ref{nonexpansivity})\textbf{ }(applied to $w\longmapsto
J(w)=A(\frac{w}{s-t})$),%
\begin{align*}
\lim_{s\rightarrow t}\int_{\Omega}A\left(  \frac{Ru(s)-Ru(t)}{s-t}\right)
\mathrm{d}x  &  \leq\lim_{s\rightarrow t}\int_{\Omega}A\left(  \frac
{u(s)-u(t)}{s-t}\right)  \mathrm{d}x\\
&  =\int_{\Omega}A\left(  \lim_{s\rightarrow t}\frac{u(s)-u(t)}{s-t}\right)
\mathrm{d}x=\psi\left(  \frac{\mathrm{d}}{\mathrm{d}t}u(t)\right)  \text{.}%
\end{align*}
The above computations hold true for a.a. $t\in(0,T)$. Thanks to the upper
bound in (\ref{def alpha}), we have $|A(s)|\leq|\alpha(s)s|\leq C(|s|^{p}+1)$
and hence, by applying the Dominated Convergence Theorem, we get
\[
\int_{0}^{T}\mathrm{e}^{-t/\varepsilon}\psi\left(  \frac{\mathrm{d}%
}{\mathrm{d}t}(Ru(t))\right)  \mathrm{d}t\leq\int_{0}^{T}\mathrm{e}%
^{-t/\varepsilon}\psi\left(  \frac{\mathrm{d}}{\mathrm{d}t}u(t)\right)
\mathrm{d}t\text{ for all }T>0\text{.}%
\]
This yields (R2.2). We readily check that (\ref{1 approach 2})\textbf{ }is
satisfied. In particular, in the case of rearrangements maps $R$, we have
$R\mathrm{d}_{V}\varphi^{2}(v)=\mathrm{d}_{V}\varphi^{2}(v)$ for all $v\in
V_{R}$. Thus, condition (\ref{1 approach 2}) follows from the well known
rearrangement inequality (see the appendix or \cite{Ka}):%
\[
\int_{\Omega}Ra\cdot Rb\geq\int_{\Omega}a\cdot b\text{ for all }a\in
L^{p}(\Omega)\text{, }b\in L^{p}(\Omega)\text{, }a,b\geq0\text{.}%
\]
A direct application of Theorem \ref{abstract invariance principle} and
Corollary \ref{coroll composition} yields the first part of Theorem
\ref{thm example 1}.

To prove the second part, we aim at applying Theorem \ref{comp princ 2}. To
this end, we now verify condition (\ref{key assumption comparison}). Note
that, for all $u,v\in W^{1,p}(0,T;L^{p}(\Omega))\cap L^{m}(0,T;X)$, one has
$u\vee v,u\wedge v\in W^{1,p}(0,T;L^{p}(\Omega))\cap L^{m}(0,T;X)$.
Furthermore, the following relations hold true a.e. in $\Omega\times(0,T)$
\begin{align*}
\left(  u\vee v\right)  ^{\prime}  &  =\left\{
\begin{array}
[c]{cc}%
u^{\prime} & \text{if }u\geq v\text{,}\\
v^{\prime} & \text{if }u<v\text{,}%
\end{array}
\right. \\
\left(  u\wedge v\right)  ^{\prime}  &  =\left\{
\begin{array}
[c]{cc}%
v^{\prime} & \text{if }u\geq v\text{,}\\
u^{\prime} & \text{if }u<v\text{,}%
\end{array}
\right. \\
\nabla\left(  u\vee v\right)   &  =\left\{
\begin{array}
[c]{cc}%
\nabla u & \text{if }u\geq v\text{,}\\
\nabla v & \text{if }u<v\text{,}%
\end{array}
\right. \\
\nabla\left(  u\wedge v\right)   &  =\left\{
\begin{array}
[c]{cc}%
\nabla v & \text{if }u\geq v\text{,}\\
\nabla u & \text{if }u<v\text{,}%
\end{array}
\right. \\
\left(  u\vee v\right)  _{|\partial\Omega}  &  =\left\{
\begin{array}
[c]{cc}%
u_{|\partial\Omega} & \text{if }u\geq v\text{,}\\
v_{|\partial\Omega} & \text{if }u<v\text{,}%
\end{array}
\right. \\
\left(  u\wedge v\right)  _{|\partial\Omega}  &  =\left\{
\begin{array}
[c]{cc}%
v_{|\partial\Omega} & \text{if }u\geq v\text{,}\\
u_{|\partial\Omega} & \text{if }u<v\text{,}%
\end{array}
\right.
\end{align*}
where $w_{|\partial\Omega}$ denotes the trace of $w$ on $\partial\Omega$.
Moreover,
\begin{align*}
&  I_{\varepsilon}(u\vee v)+I_{\varepsilon}(u\wedge v)\\
&  =\int\int_{\Omega\times\left(  0,T\right)  \cap\{u\geq v\}}\mathrm{e}%
^{-t/\varepsilon}G(u)+\int\int_{\Omega\times\left(  0,T\right)  \cap
\{u<v\}}\mathrm{e}^{-t/\varepsilon}G(v)\\
&  +\int\int_{\Omega\times\left(  0,T\right)  \cap\{u\geq v\}}\mathrm{e}%
^{-t/\varepsilon}G(v)+\int\int_{\Omega\times\left(  0,T\right)  \cap
\{u<v\}}\mathrm{e}^{-t/\varepsilon}G(u)\\
&  +\int\int_{\partial\Omega\times\left(  0,T\right)  \cap\{u\geq
v\}}\mathrm{e}^{-t/\varepsilon}\frac{1}{2b}|u|^{2}+\int\int_{\partial
\Omega\times\left(  0,T\right)  \cap\{u<v\}}\mathrm{e}^{-t/\varepsilon}%
\frac{1}{2b}|v|^{2}\\
&  +\int\int_{\partial\Omega\times\left(  0,T\right)  \cap\{u\geq
v\}}\mathrm{e}^{-t/\varepsilon}\frac{1}{2b}|u|^{2}+\int\int_{\partial
\Omega\times\left(  0,T\right)  \cap\{u<v\}}\mathrm{e}^{-t/\varepsilon}%
\frac{1}{2b}|u|^{2}\\
&  =I_{\varepsilon}^{1}(u)+I_{\varepsilon}^{1}(v),
\end{align*}
where $G(u)=\varepsilon\alpha(u^{\prime})+{\frac{1}{m}B|\nabla u|^{m}+\frac
{1}{m}C|u|^{m}-}D\frac{1}{q}|u|^{q}-hu$. By applying Theorem
\ref{abstract invariance principle}, we conclude the proof of Theorem
\ref{thm example 1}.
\end{proof}

\subsection{Fractional heat equation}

We consider the \textit{fractional heat equation }%
\begin{align}
u_{t}+\left(  -\Delta\right)  ^{s}u +\gamma u  &  =g\text{ in }\Omega
\times(0,T)\text{,}\label{fractional parabolic}\\
u  &  =0\text{ in }\left(  \mathbb{R}^{d}\setminus\Omega\right)
\times(0,T)\text{,}\label{frac bc}\\
u(0)  &  =u_{0}\text{ a.e. in }\mathbb{R}^{d}\text{,}
\label{fractional parabolic2}%
\end{align}
where $\Omega\subset\mathbb{R}^{d}$ bounded with Lipschitz boundary,
$\gamma>0$, $s\in(0,1)$, $u_{0}\in H_{0}^{s}(\Omega)$, and $g\in L^{2}%
(\Omega)$. Here, $\left(  -\Delta\right)  ^{s}$ denotes the $s$%
\textit{-fractional Laplace operator }\cite{DNPV}.

Before stating the main result of this section, let us first recall some
definitions and known results. For every $s\in(0,1)$ and $d\in\mathbb{N}$, we
denote by $H^{s}(\mathbb{R}^{d})=W^{s,2}(\mathbb{R}^{d})$ the usual
$s$-fractional Sobolev space equipped with the norm
\begin{align*}
\left\Vert u\right\Vert _{H^{s}(\mathbb{R}^{d})}^{2}  &  =\left\Vert
u\right\Vert _{L^{2}(\mathbb{R}^{d})}^{2}+[u]_{\mathbb{R}^{d},s}^{2}\\
&  =\left\Vert u\right\Vert _{L^{2}(\mathbb{R}^{d})}^{2}+\int_{\mathbb{R}^{d}%
}\int_{\mathbb{R}^{d}}\frac{|u(x)-u(y)|^{2}}{|x-y|^{d+2s}}\mathrm{d}%
x\mathrm{d}y.
\end{align*}
We use the notation
\[
H_{0}^{s}(\Omega)=\{u\in H^{s}(\mathbb{R}^{d}):u=0\text{ a.e. in }%
\mathbb{R}^{d}\setminus\Omega\}.
\]
Let $Q=\mathbb{R}^{2d}\setminus\left(  \mathbb{R}^{d}\setminus\Omega\right)
\times\left(  \mathbb{R}^{d}\setminus\Omega\right)  $. Then, the space
$H_{0}^{s}(\Omega),$ equipped with the norm
\begin{align*}
\left\Vert u\right\Vert _{H_{0}^{s}(\Omega)}^{2}  &  =\left\Vert u\right\Vert
_{L^{2}(\Omega)}^{2}+\iint_{Q}\frac{|u(x)-u(y)|^{2}}{|x-y|^{d+2s}}%
\mathrm{d}x\mathrm{d}y\\
&  =\left\Vert u\right\Vert _{L^{2}(\Omega)}^{2}+[u]_{\mathbb{R}^{d},s}^{2}%
\end{align*}
and with the scalar product
\[
(u,v)_{H_{0}^{s}(\Omega)}=(u,v)_{L^{2}(\Omega)}+\iint_{Q}\frac{\left(
u(x)-u(y)\right)  \left(  v(x)-v(y)\right)  }{|x-y|^{d+2s}}\mathrm{d}%
x\mathrm{d}y\text{,}%
\]
is a Hilbert space \cite[Lemma 7]{Se-Va}. Note that the functional $\varphi:
H^{s}_{0} (\Omega) \to(H^{s}_{0} (\Omega))^{\ast}$ defined by
\[
u \mapsto\varphi(u)=\frac{1}{2}\iint_{Q}\frac{|u(x)-u(y)|^{2}}{|x-y|^{d+2s}%
}\mathrm{d}x\mathrm{d}y
\]
is Fr\'{e}chet differentiable over $H_{0}^{s}(\Omega)$ and
\[
\left\langle \mathrm{d}_{H_{0}^{s}(\Omega)}\varphi(u),v\right\rangle
_{H_{0}^{s}\left(  \Omega\right)  } = \iint_{Q}\frac{\left(  u(x)-u\left(
y\right)  \right)  \left(  v(x)-v\left(  y\right)  \right)  }{|x-y|^{d+2s}%
}\mathrm{d}x\mathrm{d}y
\]
for every $v\in H_{0}^{s}(\Omega)$ \cite{Se-Va}. We define the fractional
Laplacian operator as \cite{DNPV}
\[
-(-\Delta)^{s}u(x)=\int_{\mathbb{R}^{d}}\frac{u(x+y)+u(x-y)-2u(x)}{|y|^{d+2s}%
}\mathrm{d}y\text{, }x\in\mathbb{R}^{d}%
\]
and we note that $\partial_{L^{2} (\Omega)} \varphi(u) = (-\Delta)^{s} u$ for
all $u \in D(\partial_{L^{2} (\Omega)} \varphi)$.

Thanks to the above preparation, we can rewrite equation
(\ref{fractional parabolic})-(\ref{fractional parabolic2}) in the gradient
flow form%
\begin{align}
u_{t}+\partial_{V}\varphi^{1}(u)-\mathrm{d}_{V}\varphi^{2}(u)  &  =g\text{ in
}V^{\ast}\text{, a.e. in }(0,T)\text{,}\label{gf frac}\\
u(0)  &  =u_{0}\text{,} \label{gf frac2}%
\end{align}
where
\begin{align*}
X  &  =H_{0}^{s}\left(  \Omega\right)  \text{, \ \ \ \ \ \ \ \ \ }%
V=L^{2}(\Omega)\text{,}\\
\varphi^{1}(u)  &  =\frac{\gamma}{2}\left\Vert u\right\Vert _{L^{2}(\Omega
)}^{2}+\frac{1}{2}\iint_{Q}\frac{|u(x)-u(y)|^{2}}{|x-y|^{d+2s}}\mathrm{d}%
x\mathrm{d}y\text{,}\\
\varphi^{2}(u)  &  = 0 \text{, \ \ \ \ \ \ \ \ }f(u)=g.
\end{align*}
The well posedness of problem (\ref{gf frac})-(\ref{gf frac2}) follows from
the classical theory of \cite{Br}. We observe that the (unique) solution to
the problem (\ref{gf frac})-(\ref{gf frac2}) in the sense of Definition
\ref{solution to nonconvex target} is a function $u\in L^{2}(0,T;H_{0}%
^{s}(\Omega))\cap H^{1}(0,T;L^{2}(\Omega))$ such that $(-\Delta)^{s} u(t)\in
L^{2}(\Omega)$ for a.e. $t\in(0,T)$ that solves equation
(\ref{fractional parabolic}) a.e. in $\Omega\times(0,T)$.

Aiming at applying Theorem \ref{abstract invariance principle} to prove
qualitative properties of the solution of (\ref{fractional parabolic}%
)-(\ref{fractional parabolic2}), we now introduce some maps $R:L^{2}\left(
\Omega\right)  \rightarrow L^{2}(\Omega)$ which describe qualitative
properties and we fix some compatibility conditions for the data.

\begin{enumerate}
\item \textit{Linear rigid transformation of the space: }$Ru(x)=u(rx)$ for
some $r\in SL(d,\mathbb{R})$, $r\Omega=\Omega$, and $g$ is $R$-invariant;

\item \textit{Symmetric decreasing rearrangement} or \textit{Schwartz
symmetrization} \cite{Ka}:\textit{ }$Ru=\left(  u^{+}\right)  ^{\ast}$,
$\Omega$ is radially symmetric, $g=\left(  g^{+}\right)  ^{\ast}$ a.e. in
$\Omega$;

\item \textit{Positive part:} $R(u)=u^{+}$ and $g\geq0$ a.e. in $\Omega$;

\item \textit{Negative part:} $R(u)=-u^{-}$ and $g\leq0$ a.e. in $\Omega$.
\end{enumerate}

By applying Theorem \ref{abstract invariance principle}, Corollary
\ref{coroll composition}, and Theorem \ref{comp princ 2}, we get the following.

\begin{theorem}
[Fractional heat equation]\label{thm example 2}\label{thm example 4}Let
$R_{i}$, $i=1,...,k$, be any collection of the maps defined above. Then, for
every $u_{0}\in H_{0}^{s}\left(  \Omega\right)  $ such that $R_{i}u_{0}=u_{0}$
for $i=1,...,k$, the strong solution $u$ to \emph{(\ref{fractional parabolic}%
)-(\ref{fractional parabolic2})} fulfills $R_{1}\circ...\circ R_{k}u=u$ a.e.
in $\Omega\times(0,T)$. Moreover, let $u$ and $v$ be the two strong solutions
to \emph{(\ref{fractional parabolic})-(\ref{frac bc})} corresponding to
initial conditions $u(0)=u_{0}$ and $v(0)=v_{0}$, with $u_{0}\leq v_{0}$ a.e.
in $\Omega$. Then, $u\leq v$ a.e. in $\Omega\times\lbrack0,T]$.
\end{theorem}

As already mentioned in the previous section, in the case $k\geq2$, assumption
$R_{i}u_{0}=u_{0}$ for all $i=1,...,k$ implies that the compatibility
conditions associated with any of the maps $R_{i}$ are satisfied\ simultaneously.

\begin{proof}
Taking advantage of the above preparation, we readily check conditions (R1),
(\ref{1 approach}) (and hence (R2)) (see the appendix or \cite{Pa} for the
case of the symmetric decreasing rearrangement). Thus, the first part of
Theorem \ref{thm example 2} follows directly from Theorem
\ref{abstract invariance principle} and Corollary \ref{coroll composition}.

In order to prove the second part of Theorem \ref{thm example 2}, we now check
that assumptions of Theorem \ref{comp princ 2} are satisfied. For all $u,v\in
H^{s}(\mathbb{R}^{d})$, one has%
\begin{align*}
\lbrack u\vee v]_{\mathbb{R}^{d},s}^{2}+[u\wedge v]_{\mathbb{R}^{d},s}^{2}  &
=\int_{\mathbb{R}^{d}}\left(  \int_{\mathbb{R}^{d}}\frac{|\left(  u\vee
v\right)  (x)-\left(  u\vee v\right)  (y)|^{2}}{|x-y|^{d+2s}}\mathrm{d}%
y\right)  \mathrm{d}x\\
&  +\int_{\mathbb{R}^{d}}\left(  \int_{\mathbb{R}^{d}}\frac{|\left(  u\wedge
v\right)  (x)-\left(  u\wedge v\right)  (y)|^{2}}{|x-y|^{d+2s}}\mathrm{d}%
y\right)  \mathrm{d}x\\
&  =A_{1}+A_{2}+A_{3}+A_{4}%
\end{align*}
where,
\begin{align*}
A_{1}  &  =\int_{u\geq v}\left(  \int_{u\geq v}\left(  \frac{|u(x)-u(y)|^{2}%
}{|x-y|^{d+2s}}+\frac{|v(x)-v(y)|^{2}}{|x-y|^{d+2s}}\right)  \mathrm{d}%
y\right)  \mathrm{d}x\text{,}\\
A_{2}  &  =\int_{u<v}\left(  \int_{u<v}\left(  \frac{|v(x)-v(y)|^{2}%
}{|x-y|^{d+2s}}+\frac{|u(x)-u(y)|^{2}}{|x-y|^{d+2s}}\right)  \mathrm{d}%
y\right)  \mathrm{d}x\text{,}\\
A_{3}  &  =\int_{u\geq v}\left(  \int_{u<v}\left(  \frac{|u(x)-v(y)|^{2}%
}{|x-y|^{d+2s}}+\frac{|v(x)-u(y)|^{2}}{|x-y|^{d+2s}}\right)  \mathrm{d}%
y\right)  \mathrm{d}x\text{,}\\
A_{4}  &  =\int_{u<v}\left(  \int_{u\geq v}\left(  \frac{|v(x)-u(y)|^{2}%
}{|x-y|^{d+2s}}+\frac{|u(x)-v(y)|^{2}}{|x-y|^{d+2s}}\right)  \mathrm{d}%
y\right)  \mathrm{d}x.
\end{align*}
We now prove that
\begin{equation}
A_{3}\leq\int_{u\geq v}\left(  \int_{u<v}\left(  \frac{|u(x)-u(y)|^{2}%
}{|x-y|^{d+2s}}+\frac{|v(x)-v(y)|^{2}}{|x-y|^{d+2s}}\right)  \mathrm{d}%
y\right)  \mathrm{d}x\text{.} \label{a3}%
\end{equation}
To this aim, let us denote $a=u(x)-v(x)$, $b=v(x)-v(y)$, $c=v(y)-u(y)$. Note
that, as $u(x)\geq v(x)$ and $u(y)<v(y)$ a.e. over the integration domain,
$ac\geq0$. Thus, (\ref{a3}) follows by a direct application of inequality
$(a+b)^{2}+(b+c)^{2}\leq b^{2}+(a+b+c)^{2}$. Similarly, we can prove%
\[
A_{4}\leq\int_{u<v}\left(  \int_{u\geq v}\left(  \frac{|v(x)-v(y)|^{2}%
}{|x-y|^{d+2s}}+\frac{|u(x)-u(y)|^{2}}{|x-y|^{d+2s}}\right)  \mathrm{d}%
y\right)  \mathrm{d}x\text{.}%
\]
Combining these estimates, we get
\[
A_{1}+A_{2}+A_{3}+A_{4}\leq\lbrack u]_{\mathbb{R}^{d},s}^{2}+[v]_{\mathbb{R}%
^{d},s}^{2}\text{.}%
\]
In particular, $u\wedge v,u\vee v\in H^{s}(\mathbb{R}^{d})$ for every $u,v\in
H^{s}(\mathbb{R}^{d})$ and conditions (\ref{min w1p}) and
(\ref{key assumption comparison}) are fulfilled.

Finally, note that the spaces $L^{2}(\Omega)$ and $H_{0}^{s}(\Omega)$,
$s\in(0,1)$ satisfy condition (\ref{clos cond 2}). Thus, the second assertion
in Theorem \ref{thm example 2} follows directly from Theorem
\ref{comp princ 2}.
\end{proof}

\subsection{Systems of reaction-diffusion equations}

We consider the \textit{diffusive Lotka-Volterra prey-predator system} given
by%
\begin{align}
u_{t}-D_{1}\Delta u  &  =Au\left(  1-\frac{u}{K}\right)  -\frac{Buv}%
{1+Ev}-F_{1}u\text{ in }\Omega\times(0,T)\text{,}\label{lotka 1}\\
v_{t}-D_{2}\Delta v  &  =\frac{Cuv}{1+Ev}-F_{2}v\text{ in }\Omega
\times(0,T)\text{,}\\
\frac{\partial u}{\partial n}  &  =\frac{\partial v}{\partial n}=0\text{ on
}\partial\Omega\times(0,T)\text{,}\\
v(0)  &  =v_{0}\text{, }u(0)=u_{0}\text{ in }\Omega\text{,} \label{lotka 2}%
\end{align}
where $A,K,D_{1},D_{2},F_{1},F_{2}>0$ and $B,C,E\geq0$ are constants and
$\Omega\subset\mathbb{R}^{d}$ is bounded with Lipschitz boundary
$\partial\Omega$. The model describes the evolution of two interacting
populations \cite{Mu, Du1, Du2}. Here $u$ and $v$ denote the concentrations of
a prey species and a predator species respectively, $D_{1},D_{2}$, and
$F_{1},F_{2}$ are the diffusion rates and the spontaneous-death rates of preys
and predators respectively. The parameters $C,B$ describe the interaction
rates of the two species while $E$ measures the so-called predator satiation
\cite{Mu, Du1, Du2}. Finally, $A$ represents the preys' birth rate (at
predators low density) and $K$ the so-called carrying capacity of the environment.

Note that negative values of $u$ and $v$ or values of $u$ larger than $K$ are
meaningless from the biological viewpoint. By applying Theorem
\ref{abstract invariance principle} together with the choice $R(u,v)=\left(
\left(  \min\{u,K\}\right)  ^{+},v^{+}\right)  $, we can prove the existence
of solutions to system (\ref{lotka 1})-(\ref{lotka 2}) starting from initial
data $\left(  u_{0},v_{0}\right)  \in\lbrack0,K]\times\lbrack0,\infty)$ a.e.
in $\Omega\times\Omega$, satisfy the same bounds at any time. To this end, we
first reformulate system (\ref{lotka 1})-(\ref{lotka 2}) in the abstract form
(\ref{nonconvex target equation})-(\ref{ic}) by defining%
\begin{align*}
V  &  =L^{2}(\Omega)\times L^{2}(\Omega)\text{,}\\
X  &  =H^{1}(\Omega)\times H^{1}(\Omega)\text{,}\\
\varphi^{1}(u,v)  &  =\frac{1}{2}\int_{\Omega}D_{1}|\nabla u|^{2}+D_{2}|\nabla
v|^{2}+F_{1}|u|^{2}+F_{2}|v|^{2}\text{,}\\
\psi(u,v)  &  =\frac{1}{2}\int_{\Omega}|u|^{2}+|v|^{2}\text{,}\\
\varphi^{2}(u)  &  =0\text{,}\\
f(u,v)  &  =\left(  AU\left(  1-\frac{U}{K}\right)  -\frac{BUV}{1+EV}%
,\frac{CUV}{1+EV}\right)  ,
\end{align*}
where $U:=(\min\{u,K\})^{+}$ and $V=v^{+}$. Note that in case $\left(
u,v\right)  $ solves (\ref{nonconvex target equation})-(\ref{ic}) in the sense
of Definition \ref{solution to nonconvex target} and $0\leq u\leq K$, $v\geq
0$, then,
\begin{align*}
&  u,v\in H^{1}(0,T;L^{2}(\Omega))\cap L^{2}(0,T;H^{1}(\Omega))\text{,}\\
&  \Delta u(t),\Delta v(t)\in L^{2}(\Omega)\text{ for a.e. }t\in(0,T)\text{,}%
\end{align*}
and $(u,v)$ fulfills identities (\ref{lotka 1})-(\ref{lotka 2}) pointwise a.e.
in $\Omega\times(0,T)$. Indeed, $f(u,v)=\left(  Au\left(  1-\frac{u}%
{K}\right)  -\frac{Buv}{1+Ev},\frac{Cuv}{1+Ev}\right)  $ if $0\leq u\leq K$
and $v\geq0$.

\begin{theorem}
[System of reaction-diffusion equations]\label{thm example 2bis}For all
$u_{0},v_{0}\in H^{1}(\Omega)$ such that $0\leq u_{0}\leq K$ and $v_{0}\geq0$
a.e. in $\Omega$, system \emph{(\ref{lotka 1})-(\ref{lotka 2})} admits a
strong solution $(u,v)$ such that $0\leq u\leq K\ $and $v\geq0$ a.e. in
$\Omega\times(0,T)$.
\end{theorem}

\begin{proof}
Define $R(u,v)=(U,V)=\left(  \left(  \min\{u,K\}\right)  ^{+},v^{+}\right)  $.
It is standard matter to check that assumptions of Proposition
\ref{wed approach nonpotential} are satisfied \cite{Me}. Moreover, (R1) can be
easily proved. We now verify condition (\ref{2 approach}) (and, thus, (R2)).
Note that
\begin{align*}
&  \int_{0}^{T}\mathrm{e}^{-t/\varepsilon}\left(  \varepsilon\psi\left(
\frac{\mathrm{d}}{\mathrm{d}t}R\left(  u,v\right)  \right)  +\varphi
^{1}\left(  R(u,v)\right)  \right)  \mathrm{d}t\\
&  \leq\int_{0}^{T}\mathrm{e}^{-t/\varepsilon}\left(  \varepsilon\psi\left(
\frac{\mathrm{d}}{\mathrm{d}t}\left(  u,v\right)  \right)  +\varphi^{1}\left(
u,v\right)  \right)  \mathrm{d}t.
\end{align*}
Let us prove that, for all $(u,v)\in L^{2}(\Omega)\times L^{2}(\Omega)$,
\begin{equation}
\left\langle f(u,v),R(u,v)\right\rangle _{L^{2}(\Omega)\times L^{2}(\Omega
)}\geq\left\langle f(u,v),(u,v)\right\rangle _{L^{2}(\Omega)\times
L^{2}(\Omega)}\text{,} \label{conto}%
\end{equation}
i.e.,
\begin{align*}
&  \int_{\Omega}AU\left(  1-\frac{U}{K}\right)  U-\frac{BUV}{1+EV}U+\frac
{CUV}{1+EV}V\\
&  \geq\int_{\Omega}AU\left(  1-\frac{U}{K}\right)  u-\frac{BUV}{1+EV}%
u+\frac{CUV}{1+EV}v\text{.}%
\end{align*}
Note that $AU\left(  1-\frac{U}{K}\right)  U=AU\left(  1-\frac{U}{K}\right)
u$ a.e. in $\Omega$. Moreover,%
\[
U^{2}=\left\{
\begin{array}
[c]{cc}%
0 & \text{if }u<0\\
u^{2} & \text{if }u\in\lbrack0,K]\\
K^{2} & \text{if }u>K
\end{array}
\right\}  \leq\left\{
\begin{array}
[c]{cc}%
0 & \text{if }u<0\\
u^{2} & \text{if }u\in\lbrack0,K]\\
Ku & \text{if }u>K
\end{array}
\right\}  =Uu\text{ a.e. in }\Omega\text{.}%
\]
Thus, as $V\geq0$,%
\[
-\frac{BUV}{1+EV}U\geq-\frac{BUV}{1+EV}u\text{ a.e. in }\Omega\text{.}%
\]
Finally, $V^{2}=Vv$ a.e in $\Omega$, which implies
\[
\frac{CUV}{1+EV}V=\frac{CUV}{1+EV}v\text{ a.e. in }\Omega\text{. }%
\]
Combining these estimates, we get (\ref{conto}). Thus, $I_{\varepsilon
,f(u,v)}(R(u,v))\leq I_{\varepsilon,f(u,v)}(u,v)$, which yields
(\ref{2 approach}). Hence, Theorem \ref{thm example 2bis} follows from a
direct application of Theorem \ref{abstract invariance principle}.
\end{proof}

\begin{remark}
\emph{Analogous systems with nonquadratic dissipation and energy functionals
of the form }%
\begin{align*}
\varphi^{1}(u,v)  &  =\frac{1}{m}\int_{\Omega}D_{1}|\nabla u|^{m}+D_{2}|\nabla
v|^{m}+F_{1}|u|^{m}+F_{2}|v|^{m}\text{,}\\
\psi(u,v)  &  =\frac{1}{p}\int_{\Omega}|u|^{p}+|v|^{p}\text{, \ \ \ \ \ \ \ }%
m\in(1,\infty)\text{, }p\in(2,\infty)
\end{align*}
\emph{can be treated in a similar way (see \cite{Ak-Me}). The argument may be
easily generalized also to systems with nonconstant spatially-dependent
coefficients.}
\end{remark}

\section{More examples \label{section other}}

The WED variational procedure and its analogous for hyperbolic problems, the
\textit{weighted-inertia-energy-dissipation} (WIDE) procedure, have been
applied to a larger class of problems including rate-independent systems
\cite{Mi-Or} and some hyperbolic problems \cite{Li-St,Se-Ti,St,Li-St2}. In
this section we use these variational approaches to prove a comparison
principle for a rate-independent system and to check qualitative properties of
solutions of a nonlinear wave equation, and of a lagrangian-mechanics system.
It is worth mentioning that the results presented in this section can be
widely generalized. In particular, an abstract theory for rate-independent
systems in Banach spaces may be developed in the spirit of Section
\ref{sec qual prop} and Section \ref{section comparison}. This, however, is
beyond our scope. Moreover, the (relatively simple) examples we present here
suffice to provide the main ideas and a guide line for the developing of
abstract results in the spirit of what we have done above for doubly-nonlinear problems.

\subsection{Rate-independent systems \label{rate ind section}}

In this section we prove a comparison principle for \emph{energetic solutions}
to the following rate-independent inclusion
\begin{align}
\mathrm{sign}(u^{\prime})+\tilde{\phi}^{\prime}(u)-a \Delta u  &  \ni
h(t)\text{ in }\Omega\times(0,T)\text{,}\label{rate ind}\\
\frac{\partial u}{\partial n}  &  =0\text{ on }\partial\Omega\text{,}\\
u(0)  &  =u_{0}\text{,} \label{rate ind ic}%
\end{align}
where $\Omega\subset\mathbb{R}^{d}$ is bounded with Lipschitz boundary
$\partial\Omega$ with outward normal unit vector $n$, $u_{0}\in H^{1}(\Omega
) \cap L^p (\Omega)$ for some $p \geq 2$, 
$a \geq0$, $\tilde{\phi}\in C^{1}(\mathbb{R})$ is assumed to be convex and satisfying 
$$
\phi (u) \leq C(|u|^p+1) \text{ for some positive constant } C \text{ and all } u\in \mathbb{R},
$$
and $h\in L^{\infty}(0,T;L^{2}(\Omega))$. Aiming at applying the WED theory
for rate-independent problems developed in \cite{Mi-Or}, we start by rewriting
inclusion (\ref{rate ind})-(\ref{rate ind ic}) in the form
\begin{equation}
0\in\partial_{L^{2}(\Omega)}\psi(u^{\prime})+\partial_{L^{2}(\Omega)}%
\phi(u)\text{,} \label{abstract rate-indep}%
\end{equation}
where
\[
\phi(u)= \left\{
\begin{array}
[c]{cc}%
\displaystyle\int_{\Omega}\tilde{\phi}(u)+\frac{a}{2}|\nabla u|^{2}-hu &
\text{if }u\in H^{1}(\Omega)\text{, }\tilde{\phi}(u)\in L^{1}(\Omega)\text{,}\\
+\infty & \text{else,}%
\end{array}
\right.
\]
$\psi(v)=\int_{\Omega}|v|$, and
\[
\partial_{L^{2}(\Omega)}\psi(v)(x)=\partial_{\mathbb{R}}|v|(x)=\mathrm{sign}%
(v(x))=\left\{
\begin{array}
[c]{cc}%
\{-1\} & \text{for }v(x)\in\lbrack-\infty,0),\\
\lbrack-1,1] & \text{for }v(x)=0,\\
\{1\} & \text{for }v(x)\in(0,+\infty],
\end{array}
\right.
\]
for a.e. $x\in\Omega$. The abstract inclusion (\ref{abstract rate-indep})
arise ubiquitously in applications, from mechanics and electromagnetism to
economics (see, e.g, \cite{Mi-Ru}). An elliptic operator as in (\ref{rate ind}%
) appears frequently in models concerning micromagnetics and plasticity.

The notion of \emph{energetic solutions} to rate-independent systems is given
by the following definition.

\begin{definition}
[Energetic solution]We define $u\in\mathrm{BV}([0,T];L^{2}(\Omega))$
\emph{energetic solution} to the rate independent problem
\emph{(\ref{rate ind})-(\ref{rate ind ic})} if it satisfies%
\begin{align*}
\phi(u(t))  &  \leq\phi(w)+\psi(w-u(t))\text{ for all }w\in L^{2}%
(\Omega)\text{ and a.e. }t\in\lbrack0,T]\text{,}\\
\phi(u(t))+\int_{0}^{t}\psi(\mathrm{d}u)  &  =\phi(u_{0})\text{ for a.e. }%
t\in\lbrack0,T]\text{,}%
\end{align*}
where $\int_{0}^{t}\psi(\mathrm{d}u)$ is defined by
\[
\int_{0}^{T}\psi(\mathrm{d}u)=\sup\left\{  \sum_{j=1}^{N}\psi(u(s_{j}%
)-u(s_{j-1})):N\in\mathbb{N},~0\leq s_{1}<...<s_{N}\leq T\right\}  .
\]

\end{definition}

Existence of energetic solutions to (\ref{rate ind})-(\ref{rate ind ic}) is
classical and a proof can be found, e.g., in \cite{Mi-Ru}. We remark that
solutions are in general not unique.

Our technique is based on the WED approach to rate-independent problems
studied in \cite{Mi-Or}. Thus, before stating our comparison principle, we
sketch the results in \cite{Mi-Or} for the reader's convenience. For every
$\varepsilon>0$ small enough, the functional $I_{\varepsilon}$ defined by
\[
I_{\varepsilon}(u)=\left\{  \displaystyle{%
\begin{array}
[c]{cc}%
\mathrm{e}^{-T/\varepsilon}\phi(u(T))+\int_{0}^{T}\mathrm{e}^{-t/\varepsilon
}\varepsilon\psi(\mathrm{d}u) & \\
+\int_{0}^{T}\mathrm{e}^{-t/\varepsilon}\phi(u(t))\mathrm{d}t & \text{if }u\in
K(u_{0})\text{,}\\
+\infty & \text{else,}%
\end{array}
}\right.
\]
admits a minimizer $u_{\varepsilon}$ over
\begin{equation}
K(u_{0})=\{u\in\mathrm{BV}([0,T];L^{2}(\Omega)):u(0)=u_{0}\}.
\label{K rate indip}%
\end{equation}
Moreover, for every sequence $\varepsilon_{n}\rightarrow0$, there exists a
(not relabeled) subsequence $\varepsilon_{n}\rightarrow0$ such that
\begin{equation}
u_{\varepsilon_{n}}\rightarrow u\text{ a.e. in }[0,T] \label{conv rate indip}%
\end{equation}
and $u$ is an energetic solution to inclusion (\ref{rate ind}%
)-(\ref{rate ind ic}).

Taking advantage of the WED approach and arguing as in Theorem
\ref{comp princ 2}, we can prove the following comparison principle.

\begin{theorem}
[Comparison principle for rate-independent systems]%
\label{thm rate indep abstr}Let $v_{0},u_{0}\in D(\phi)$ be such that
$u_{0}\leq v_{0}$ a.e. in $\Omega$. Then, there exist two energetic solutions
$u,v$ to inclusion \emph{(\ref{rate ind})} corresponding to initial conditions
$u(0)=u_{0}$ and $v(0)=v_{0}$ such that $u\leq v$ for a.e. in $\Omega
\times(0,T)$.
\end{theorem}

\begin{proof}
For all $\varepsilon>0$ sufficiently small let $u_{\varepsilon}$ and
$v_{\varepsilon}$ be minimizers of $I_{\varepsilon}$ over $K(u_{0})$ and
$K(v_{0})$ respectively. Recalling that $w_{1}\vee w_{2},w_{1}\wedge w_{2}%
\in\mathrm{BV}([0,T];L^{2}(\Omega))$ for all $w_{1},w_{2}\in\mathrm{BV}%
([0,T];L^{2}(\Omega))$, we have that $u_{\varepsilon}\wedge v_{\varepsilon}\in
K(u_{0})$ and $u_{\varepsilon}\vee v_{\varepsilon}\in K(v_{0})$. Moreover, it
is easy to prove that
\[
I_{\varepsilon}(u_{\varepsilon}\vee v_{\varepsilon})+I_{\varepsilon
}(u_{\varepsilon}\wedge v_{\varepsilon})\leq I_{\varepsilon}(u_{\varepsilon
})+I_{\varepsilon}(v_{\varepsilon})\text{.}%
\]
Thus, by applying the abstract comparison principle given by Lemma
\ref{abstract comparison principle}, we deduce that $\tilde{u}_{\varepsilon
}:=u_{\varepsilon}\wedge v_{\varepsilon}$ and $\tilde{v_{\varepsilon}%
}:=u_{\varepsilon}\vee v_{\varepsilon}$ minimize $I_{\varepsilon}$ over
$K(u_{0})$ and $K(v_{0})$ respectively. Trivially, $\tilde{u_{\varepsilon}%
}\leq\tilde{v_{\varepsilon}}$. By using convergence (\ref{conv rate indip}),
we have (up to some not relabeled subsequences) that
\begin{align*}
\tilde{u}_{\varepsilon}  &  \rightarrow u\text{ a.e. in }[0,T]\text{,}\\
\tilde{v}_{\varepsilon}  &  \rightarrow u\text{ a.e. in }[0,T]\text{,}%
\end{align*}
and $u$ and $v$ are energetic solutions to inclusion (\ref{rate ind})
corresponding to the initial conditions $u(0)=u_{0}$ and $v(0)=v_{0}$
respectively. Moreover, we have that $u\leq v$ a.e. in $[0,T]$.
\end{proof}

Here, we have chosen to deal with a simple problem for sake of brevity and
simplicity and in order to avoid technicalities. We remark that the results
presented in this section can be generalized. In particular, under suitably
assumptions on the energy functional $\phi$ and on the dissipation potential
$\psi$, rate-independent problems on abstract Banach spaces can be treated
similarly in the spirit of Section \ref{section assumptions} (see \cite{Mi-Or}
for the WED approach).

\subsection{Nonlinear wave equation}

In this section we deal with the hyperbolic problem given by
\begin{align}
\rho u_{tt}+\nu u_{t}-\Delta u+F^{\prime}(u)  &  =0\text{ in }\Omega
\times(0,T),\label{hyp target eq}\\
u(0)  &  =u_{0}\text{, }u_{t}(0)=v_{0}, \label{ic hyp}%
\end{align}
where $\rho>0,~v\geq0$ are constants, formulated in a bounded or unbounded
domain $\Omega$ and coupled with different types of boundary conditions. We
consider initial data $u_{0},v_{0}\in H^{1}(\Omega)\cap L^{p}(\Omega)$ for
some $p\geq2$. We restrict ourself to the case of $0<T<+\infty$ for
simplicity, although the case of unbounded time intervals (i.e., $T=+\infty$)
can be treated analogously (see \cite{Se-Ti} for the WIDE procedure in this
case). We assume i) $\Omega\subset\mathbb{R}^{d}$ to be nonempty, open, and
Lipschitz and that the problem is coupled with Dirichlet or Neumann boundary
conditions, or ii) $\Omega=\mathbb{T}^{d}$, where $\mathbb{T}^{d}%
\mathbb{=[}0,2\pi)^{d}$ is the $d$-dimensional flat torous, together with
periodic boundary conditions, or iii) $\Omega=\mathbb{R}^{d}$. Moreover, let
$F\in C^{1}(\mathbb{R})$ be $\lambda$-convex for some $\lambda\in\mathbb{R}$
and let exist $C>0$ such that
\[
\frac{1}{C}|s|^{p}-C\leq F(s)\text{, }|F^{\prime}(s)|^{p^{\prime}}\leq
C(1+|s|^{p})\text{.}%
\]
Furthermore, if $\Omega=\mathbb{R}^{d}$, we ask $F(s)=|s|^{p}$ and $\nu=0$.

We are interested in weak solutions to the Cauchy problem (\ref{hyp target eq}%
)-(\ref{ic hyp}) with regularity $u\in Q$, where
\[
Q=H^{2}(0,T;L^{2}(\Omega))\cap L^{2}(0,T;X)\cap L^{p}(0,T;L^{p}(\Omega
))\text{.}%
\]
Here $X=H_{0}^{1}(\Omega)$ in the case of bounded domain $\Omega$ and Dirichlet
boundary conditions, $X=H^{1}(\mathbb{T}^{d})$ in the case of periodic
boundary conditions, and $X=H^{1}(\Omega)$ in the case of Neumann boundary
conditions or $\Omega=\mathbb{R}^{d}$. In order to prove qualitative
properties for solutions to equation (\ref{hyp target eq})-(\ref{ic hyp}), we
follow the same idea presented in the above sections. To this aim, we now
illustrate the WIDE approach to problem (\ref{hyp target eq})-(\ref{ic hyp})
for the reader's convenience. The following result was first conjectured by De Giorgi 
(in the case $\nu =0$) \cite{DeG} and than proved in 
\cite{Li-St, St} (see also \cite{Se-Ti, Se-Ti2}). 
The WIDE functional $I_{\varepsilon}:Q\rightarrow\mathbb{R}%
\cup\{\infty\}$, defined by%
\[
I_{\varepsilon}(u)=\int_{0}^{T}\int_{\Omega}e^{-t/\varepsilon}\left(
\frac{\varepsilon^{2}\rho}{2}|u^{\prime\prime}|^{2}+\frac{\varepsilon\nu}%
{2}|u^{\prime}|^{2}+\frac{1}{2}|\nabla u|^{2}+F(u)\right)  \mathrm{d}%
x\mathrm{d}t
\]
admits a unique minimizer $u_{\varepsilon}$ over the set
\[
K(u_{0},v_{0})=\{u\in Q:u(0)=u_{0},\partial_{t}u(0)=v_{0}\}
\]
for every $\varepsilon>0$ sufficiently small. Furthermore, up to (not
relabeled) subsequences
\begin{equation}
u_{\varepsilon}(t)\rightarrow u(t)\text{ pointwise a.e. in }\Omega\text{ for
all }t\in\lbrack0,T] \label{conv wave}%
\end{equation}
and the limit $u$ is a weak solution to equation (\ref{hyp target eq}%
)-(\ref{ic hyp}). We recall that uniqueness for $p$ large is an open problem.

Taking advantage of this variational procedure, we can prove some symmetries
for solutions to equation (\ref{hyp target eq})-(\ref{ic hyp}), in the spirit
of Section \ref{section examples}. To this aim, we introduce maps $R$ which
describe symmetries and we fix some compatibility conditions on the problem's data.

\begin{enumerate}
\item \textit{Linear rigid transformation}: $Ru(x)=u(rx)$, where $r\in
SL(d,\mathbb{R})$ and $r\Omega=\Omega$.

\item \textit{Translation}: $Ru(x)=u(x+\tilde{x})$ for some $\tilde{x}%
\in\mathbb{R}^{d}$ and $\Omega=\mathbb{R}^{d}$ or $\Omega=\mathbb{T}^{d}$.

\item \textit{Averaging in the direction }$y\in\mathbb{R}^{d}$: given
$y\in\mathbb{R}^{d}$ such that $|y|=1$, decompose every $x\in\mathbb{R}^{d}$
as $x=\alpha y+x_{2}$ where $\alpha\in\mathbb{R}$ and $x_{2}$ is orthogonal to
$y$. Assume $\Omega=\{\alpha y+x_{2}\in\mathbb{R}^{d}:\alpha\in(0,L)$,
$x_{2}\in\Omega^{\prime}\}$, where $L>0$ and $\Omega^{^{\prime}}$ is a subset
of $\mathbb{R}^{d-1}$ (note that $\Omega=\mathbb{T}^{d}$ satisfies this
assumption with $L=2\pi$ and $\Omega^{\prime}=\mathbb{T}^{d-1}$). Moreover,
let $F$ be convex. Define $Ru(x)=\frac{1}{L}\int_{0}^{L}u(sy+x_{2}%
)\mathrm{d}s$.
\end{enumerate}

\begin{remark}
\emph{We observe that invariance under the action of a map }$R$\emph{ as in 2)
implies periodicity in the direction of }$\frac{\tilde{x}}{|\tilde{x}|}$\emph{
with period }$|\tilde{x}|$\emph{. Functions }$u$\emph{ which are invariant
under the action of }$R$\emph{ as in 3) are instead constant in the direction
}$y$\emph{.}
\end{remark}

We now prove existence of invariant solutions to equation (\ref{hyp target eq}%
)-(\ref{ic hyp}).

\begin{theorem}
[Semilinear wave equations]\label{qual prop wave}Let $u_0,v_0 \in X$. Let $\nu\geq0$ and $R_{i}$,
$i\in\{1,...,k\}$, be any collection of the above maps. Assume $R_{i}%
u_{0}=u_{0}$ and $R_{i}v_{0}=v_{0}$ for all $i\in\{1,...,k\}$. Then, there
exists a weak solution $u$ to equation \emph{(\ref{hyp target eq}%
)-(\ref{ic hyp})} such that $u=R_{1}\circ...\circ R_{k}u$.
\end{theorem}

\begin{proof}
As a direct consequence of the assumptions on the initial data, for any
$i\in\{1,...,k\}$, we have that $R_{i}u\in K(u_{0},v_{0})$ for all $u\in
K(u_{0},v_{0})$. Moreover, one can easily prove that $I_{\varepsilon}%
(R_{i}u)\leq I_{\varepsilon}(u)$ for all $u\in K(u_{0,}v_{0})$. Let
$u_{\varepsilon}\in K(u_{0},v_{0})$ be the unique minimizer of $I_{\varepsilon
}$ over $K(u_{0},v_{0})$. By uniqueness, we deduce invariance $R_{i}%
u_{\varepsilon}=u_{\varepsilon}$ for all $i\in\{1,...,k\}$. In particular,
$u_{\varepsilon}$ is $R_{1}\circ...\circ R_{k}$-invariant. Using convergence
(\ref{conv wave}), we extract a subsequence $\varepsilon_{n}\rightarrow0$ such
that $u_{\varepsilon_{n}}\rightarrow u$ pointwise a.e. in $\Omega\times(0,T)$.
Then, for a.a. $(x,t)\in\Omega\times(0,T),$ $
u(x,t)=\lim_{\varepsilon_{n}\rightarrow0}u_{\varepsilon}(x,t)=\lim
_{\varepsilon_{n}\rightarrow0}R_{1}\circ...\circ R_{k}u_{\varepsilon
}(x,t)=R_{1}\circ...\circ R_{k}u(x,t).$
\end{proof}

\begin{remark}
\emph{Note that invariance under rearrangement transformations cannot be
expected here. Indeed, monotonicity properties do not hold true for solutions
to the wave equation. Similarly, we cannot apply the same idea to truncation
maps }$R$\emph{ of the form }$R(u)=\pm(u-M)^{+}+M$\emph{ as comparison
principles (with constant functions) are in general false for the wave
equation. Moreover, }$K(u_{0},v_{0})$\emph{ is a subset of }$H^{2}%
(0,T;L^{2}(\Omega))$\emph{. Thus, }$Ru\in K(u_{0},v_{0})$\emph{ can not be
expected for every }$u\in K(u_{0},v_{0})$\emph{ and }$R$\emph{ a rearrangement
or truncation operator.}
\end{remark}

\subsection{Lagrangian mechanics}

Consider now the Lagrangian system
\begin{align}
Mu_{tt}+\nu u_{t}+\nabla U(u)  &  =0\text{ in }(0,T)\text{, }%
\label{lagrangian target equation}\\
u(0)  &  =u_{0}\text{, }u_{t}(0)=v_{0}, \label{lagrangian ic}%
\end{align}
where $u:(0,T)\rightarrow\mathbb{R}^{d}$, $M$ is a
positive definite $d\times d$ matrix, $\nu\geq0$, and $U\in C^{1}%
(\mathbb{R}^{d},\mathbb{R})$ is convex and bounded from below. We summarize
here the WIDE approach to system (\ref{lagrangian target equation}%
)-(\ref{lagrangian ic}) studied in \cite{Li-St2}. For every $\varepsilon>0$,
the functional
\[
I_{\varepsilon}(u)=\int_{0}^{T}e^{-t/\varepsilon}\left(  \frac{\varepsilon
^{2}}{2}u^{\prime\prime}\cdot Mu^{\prime\prime}+\frac{\varepsilon\nu}%
{2}|u^{\prime}|^{2}+U(u)\right)  \mathrm{d}t
\]
admits a unique minimizer over the set $K(u_{0},v_{0})=\{u\in H^{2}%
([0,T];\mathbb{R}^{d}):u(0)=u_{0},~\partial_{t}u(0)=v_{0}\}$. Moreover, there
exists a (not relabeled) subsequence $\varepsilon\rightarrow0$ such that
$u_{\varepsilon}\rightarrow u$ pointwise a.e. in $(0,T)%
\times\mathbb{R}^{d}$ and $u$ is a strong solution to system
(\ref{lagrangian target equation})-(\ref{lagrangian ic}). We recall that
solutions to (\ref{lagrangian target equation})-(\ref{lagrangian ic}) are, in
general, not unique. Take, e.g., $d=1$ and $U(s)=\left(  s^{+}\right)  ^{3/2}$.

Taking advantage of these results, by following the idea presented in the
previous sections, we can prove existence of solutions to system
(\ref{lagrangian target equation})-(\ref{lagrangian ic}) invariant under the
action of maps $R$ defined as follows. Let $r\in M(\mathbb{R}^{d\times d})$ be
such that $r^{T}r=1$, $v\in\mathbb{R}^{d}$, and assume $U(ru+v)\leq U(u)$
(e.g., $U(u)=V(|u|)$ for some $V\in C^{1}(\mathbb{R}^{+})$ if $v=0$). Define
$Ru=ru+v$.

Arguing as in Theorem \ref{qual prop wave}, one can prove the following.

\begin{theorem}
[Lagrangian mechanics]\label{lagrangian symmetry}Let $u_{0}$, $v_{0}$, and $R$
be such that $Ru_{0}=u_{0}$ and $Rv_{0}=v_{0}$. Then, there exists a solution
$u$ to \emph{(\ref{lagrangian target equation})-(\ref{lagrangian ic})} such
that $u=Ru$.
\end{theorem}

\section{Appendix, rearrangement maps}

We recall here the definitions and some basic properties of rearrangement
maps, for the reader's convenience. For a fuller treatment, we refer the
reader to \cite{Ka}.

Rearrangement maps transform a given function $u$ into a new function
$u^{\ast}$ that has some desired property, e.g., symmetry. This is done by a
rearrangement of the level sets of the function. Thus, in order to define
rearrangement maps, we start by introducing some rearrangements of measurable
sets $\Omega\subset\mathbb{R}^{d}$ of finite Lebesgue measure.

\begin{enumerate}
\item The \emph{symmetric rearrangement} $\Omega^{\ast}=\{x\in\mathbb{R}%
^{d}:|x|<r\}$, where $r$ is such that $|\Omega|=|\Omega^{\ast}|$.

\item The \emph{symmetric rearrangement with respect to a hyperplane}. Let
$H\subset\mathbb{R}^{d}$ be a $m$-dimensional hyperplane. Then, we decompose
every $x\in\mathbb{R}^{d}$ as $x=x_{1}+x_{2}$, $x_{1}\in H$, $x_{2}\in
H^{\bot}$ and define $\Omega^{\ast,H}=\{x_{1}+x_{2}:x_{1}\in(\Omega-x_{2}\cap
H)^{\ast}$ and $x_{2}\in P^{\bot}(\Omega)\}$, where $\Omega-x_{2}%
=\{x-x_{2}:x\in\Omega\}$, $(\Omega-x_{2}\cap H)^{\ast}$ denotes the
$m$-dimensional symmetric rearrangement of $(\Omega-x_{2}\cap H)$, and
$P^{\bot}:\mathbb{R}^{d}\rightarrow H^{\bot}$ denotes the usual
orthogonal-projection map.

\item The \emph{monotone rearrangement with respect to the direction}
$y\in\mathbb{R}^{d}$. Given $y\in\mathbb{R}^{d}$ such that $|y|=1$, we
decompose every $x\in\mathbb{R}^{d}$ as $x=\alpha y+x_{2}$ where $\alpha
\in\mathbb{R}$ and $x_{2}$ is orthogonal to $y$. We define $\Omega^{\ast
,y}=\{\alpha y+x_{2}:0\leq\alpha<\mathcal{H}^{1}(\Omega\cap L(x_{2}))$ and
$x_{2}\in P^{\bot}(\Omega)\}$, where $L(x_{2})=\{\mathbb{\beta}y+x_{2}:\beta
\in\mathbb{R}\}$, $\mathcal{H}^{1}$ denotes the $1$-dimensional Hausdorff
measure and $P^{\bot}:\mathbb{R}^{d}\rightarrow\{y\}^{\bot}$ denotes the
orthogonal projection on the subspace orthogonal to $y$.
\end{enumerate}

Note that the rearrangements defined above are area preserving, i.e.,
$|\Omega|=|\Omega^{\ast}|=|\Omega^{\ast,H}|=|\Omega^{\ast,y}|$.

We now consider functions $u:\Omega\subseteq\mathbb{R}^{d}\rightarrow
\mathbb{[}0,\mathbb{\infty)}$ such that $|\{x\in\Omega:u(x)>t\}|$ is finite
for all $t>0$, and we define rearrangement maps as follows.

\begin{enumerate}
\item The \emph{symmetric decreasing rearrangement }$u^{\ast}:\Omega^{\ast
}\rightarrow\mathbb{[}0,\mathbb{\infty)}$, $u^{\ast}(x):=\sup\{c\in
\mathbb{R}:x\in\{u>c\}^\ast\}$ or equivalently $u^{\ast}(x)=\int_{0}^{\infty}%
\chi_{\{u>t\}^{\ast}}(x)\mathrm{d}t$.

\item The \emph{symmetric decreasing rearrangement with respect to the
hyperplane} $H\subset\mathbb{R}^{d}$ $u^{\ast,H}:\Omega^{\ast,H}%
\rightarrow\mathbb{[}0,\mathbb{\infty)}$, $u^{\ast}(x)=\int_{0}^{\infty}%
\chi_{\{u>t\}^{\ast,H}}(x)\mathrm{d}t$.

\item The \emph{monotone decreasing rearrangement with respect to the
direction} $y\in\mathbb{R}^{d}$ $u^{\ast,y}:\Omega^{\ast,y}\rightarrow
\mathbb{[}0,\mathbb{\infty)}$, $u^{\ast}(x)=\int_{0}^{\infty}\chi
_{\{u>t\}^{\ast,y}}(x)\mathrm{d}t$.
\end{enumerate}

We note that, by a direct consequence of the definition, rearranged functions
are measurable and lower semicontinuous. Moreover, their level sets are
rearrangements of the level sets of $u$. We now recall some known properties
of rearrangement maps.

\begin{lemma}
[Rearrangement inequalities]Let $R\Omega$ be one of the rearrangement of
the set $\Omega$ defined above and $Ru$ be the corresponding rearrangement of the function
$u:\Omega\rightarrow\lbrack0,\infty)$. Then, the following inequalities hold true.

\begin{enumerate}
\item \emph{Conservation of }$L^{p}$\emph{-norms}: $\left\Vert u\right\Vert
_{L^{p}(\Omega)}=\left\Vert Ru\right\Vert _{L^{p}(R\Omega)}$ for all
$p\in\lbrack1,\infty]$, $u\in L^{p}(\Omega)$.

\item $\int_{\Omega}uv\leq\int_{R\Omega}\left(  Ru\right)  \left(  Rv\right)
$ for all $u\in L^{p}(\Omega)$, $v\in L^{p/(p-1)}(\Omega)$, $p\in
\lbrack1,\infty]$.

\item \emph{Nonexpansivity of rearrangements}: $\int_{R\Omega}J(Ru-Rv)\leq
\int_{\Omega}J(u-v)$ for all $J:\mathbb{R}\rightarrow\mathbb{R}$ nonnegative,
convex, and such that $J(0)=0$.

\item \emph{P\'{o}lya-Szeg\"{o} inequality}: $\left\Vert \nabla u\right\Vert
_{L^{p}(\Omega)}\geq\left\Vert \nabla(Ru)\right\Vert _{L^{p}(R\Omega)}$ for
all $p\in\lbrack1,\infty]$, $u\in W^{1,p}(\Omega)$. In particular, $Ru\in
W^{1,p}(R\Omega)$ for all $u\in W^{1,p}(\Omega)$.

\item \emph{Fractional P\'{o}lya-Szeg\"{o} inequality \cite{Pa}}:
$[u]_{s,\mathbb{R}^{d}}\geq\lbrack u^{\ast}]_{s,\mathbb{R}^{d}}$ for all
$s\in(0,1)$, $u\in H^{s}(\mathbb{R}^{d})$. Here, $[u]_{s,\mathbb{R}^{d}}$
denotes the usual $s$-Gagliardo seminorm and $H^{s}(\mathbb{R}^{d}%
)=W^{s,2}(\mathbb{R}^{d})$ the usual fractional Sobolev space
\emph{\cite{DNPV}}. In particular, $u^{\ast}\in H^{s}(\mathbb{R}^{d})$ for all
$u\in H^{s}(\mathbb{R}^{d})$.
\end{enumerate}
\end{lemma}

The above inequality are well known. We refer to \cite{Ka} for a proof of 1-4
and to \cite{Pa} for a proof of 5.

\end{document}